\documentclass[11pt]{amsart}
\usepackage{amsmath}
\usepackage{amssymb}
\usepackage{amscd}
\usepackage{amsthm}

\usepackage[all]{xy}

\newtheorem{thm}{Theorem}[section]
\newtheorem{lem}[thm]{Lemma}

\newtheorem{conj}[thm]{Conjecture}
\newtheorem{prop}[thm]{Proposition}

\theoremstyle{remark}

\newtheorem{rem}[thm]{Remark}

\theoremstyle{definition}

\numberwithin{equation}{section}

\allowdisplaybreaks[4]

\DeclareMathOperator{\Pic}{Pic}

\DeclareMathOperator{\supp}{supp}

\DeclareMathOperator{\Sym}{Sym}

\DeclareMathOperator{\Proj}{Proj}

\DeclareMathOperator{\sing}{sing}
\DeclareMathOperator{\Bs}{Bs}
\DeclareMathOperator{\Span}{Span}
\DeclareMathOperator{\IM}{Im}

\begin{document}

\vfuzz0.5pc
\hfuzz0.5pc 

\newcommand{\claimref}[1]{Claim \ref{#1}}
\newcommand{\thmref}[1]{Theorem \ref{#1}}
\newcommand{\propref}[1]{Proposition \ref{#1}}
\newcommand{\lemref}[1]{Lemma \ref{#1}}
\newcommand{\coref}[1]{Corollary \ref{#1}}
\newcommand{\remref}[1]{Remark \ref{#1}}
\newcommand{\conjref}[1]{Conjecture \ref{#1}}
\newcommand{\questionref}[1]{Question \ref{#1}}
\newcommand{\defnref}[1]{Definition \ref{#1}}
\newcommand{\secref}[1]{Sec. \ref{#1}}
\newcommand{\ssecref}[1]{\ref{#1}}
\newcommand{\sssecref}[1]{\ref{#1}}

\def \red{{\mathrm{red}}}
\def \tors{{\mathrm{tors}}}
\def \EQ{\Leftrightarrow}

\def \mapright#1{\smash{\mathop{\longrightarrow}\limits^{#1}}}
\def \mapleft#1{\smash{\mathop{\longleftarrow}\limits^{#1}}}
\def \mapdown#1{\Big\downarrow\rlap{$\vcenter{\hbox{$\scriptstyle#1$}}$}}
\def \smapdown#1{\downarrow\rlap{$\vcenter{\hbox{$\scriptstyle#1$}}$}}
\def \A{{\mathbb A}}
\def \I{{\mathcal I}}
\def \J{{\mathcal J}}
\def \CO{{\mathcal O}}
\def \C{{\mathcal C}}
\def \BC{{\mathbb C}}
\def \BQ{{\mathbb Q}}
\def \m{{\mathcal M}}
\def \H{{\mathcal H}}
\def \S{{\mathcal S}}
\def \Z{{\mathcal Z}}
\def \BZ{{\mathbb Z}}
\def \W{{\mathcal W}}
\def \Y{{\mathcal Y}}
\def \T{{\mathcal T}}
\def \P{{\mathbb P}}
\def \CP{{\mathcal P}}
\def \G{{\mathbb G}}
\def \F{{\mathbb F}}
\def \BR{{\mathbb R}}
\def \D{{\mathcal D}}
\def \L{{\mathcal L}}
\def \f{{\mathcal F}}
\def \E{{\mathcal E}}
\def \BN{{\mathbb N}}
\def \N{{\mathcal N}}
\def \X{{\mathcal X}}
\def \hom{{\mathcal Hom}}

\def \closure#1{\overline{#1}}
\def \EQ{\Leftrightarrow}
\def \imply{\Rightarrow}
\def \isom{\cong}
\def \embed{\hookrightarrow}
\def \tensor{\mathop{\otimes}}
\def \wt#1{{\widetilde{#1}}}
\def \ol#1{\overline{#1}}

\title{On Vojta's $1+\varepsilon$ Conjecture}

\author{Xi Chen}

\address{632 Central Academic Building\\ 
University of Alberta\\
Edmonton, Alberta T6G 2G1, CANADA}
\email{xichen@math.ualberta.ca}

\date{Dec 7, 2006}

\maketitle

\section{Introduction}

In \cite{V1} and \cite{V2}, P. Vojta conjectured that

\begin{conj}[$1+\varepsilon$ Conjecture]\label{CONJ001}
Let $\pi: X\to B$ be a flat family of projective curves over a projective curve
$B$ with connected fibers. Suppose that $X$ has at worst quotient singularities.
Then for every $\varepsilon > 0$, there exists a constant $N_\varepsilon$ such
that
\begin{equation}\label{E001}
\omega_{X/B} \cdot C \le (1 + \varepsilon) (2g(C)-2) + N_\varepsilon (X_b \cdot C)
\end{equation}
for every irreducible curve $C\subset X$ that dominates $B$, where $\omega_{X/B}$ is the
relative dualizing sheaf of $X/B$, $X_b$ is a general fiber of $X/B$
and $g(C)$ is the geometric genus of $C$.
\end{conj}

\begin{rem}\label{REM500}
From the number-theoretical point of view, one can think of $X$ as an algebraic curve
$X_k$ over
the function field $k = K(B)$ and the multi-section $C\subset X$ as an algebraic point $p_C$
on $X_{\ol{k}} = X_k\tensor \ol{k}$. The logarithmic height $h(p_C)$ and discriminant
$\Delta(p_C)$ of $p_C$ are defined to be
\begin{equation}\label{E500}
h(p_C) = \frac{\omega_{X/B} \cdot C}{\deg(K(C)/K(B))} \text{ and }
\Delta(p_C) = \frac{2g(C) - 2}{\deg(K(C)/K(B))}
\end{equation}
respectively, where $\deg(K(C)/K(B)) = X_b\cdot C$, obviously. With these notations,
\eqref{E001} can be put in the form
\begin{equation}\label{E501}
h(p_C) \le (1+\varepsilon) \Delta(p_C) + N_\varepsilon
\end{equation}
Note that the definition of the height $h(p_C)$ depends on the choice of the birational
model $X$ of $X_k$. However, it is not hard to see that \eqref{E501} holds regardless of
the choice of the birational model (see below).
\end{rem}

Vojta proved that \eqref{E001} holds with $1+\varepsilon$ replaced by
$2+\varepsilon$. This conjecture was settled recently by K. Yamanoi \cite{Y}. M. McQuillan
later gave an algebro-geometric proof in \cite{M}. However, we find his proof quite hard to follow.
Inspired by his idea, we will give another proof of this conjecture and generalize it to
the log case. Compared to his proof, ours is more elementary.

It seems natural to study a (generalized) log version of the $1+\varepsilon$
conjecture. For a log variety $(X, D)$ and a curve $C\subset X$ that meets $D$ properly,
we define $i_X(C, D)$ to be the number of the points in $\nu^{-1}(D)$, where $\nu: \wt{C}\to
C\subset X$ is the normalization of $C$.

\begin{conj}\label{CONJ002}
Let $\pi: X\to B$ be a flat family of projective curves over a projective curve
$B$ with connected fibers. Suppose that $X$ has at worst quotient singularities and
$D\subset X$ is a reduced effective divisor on $X$.
Then for every $\varepsilon > 0$, there exists a constant $N_\varepsilon$ such that
\begin{equation}\label{E002}
(\omega_{X/B} + D) \cdot C \le (1+\varepsilon) (2g(C)-2 + i_X(C, D)) + N_\varepsilon (X_b \cdot C)
\end{equation}
for every irreducible curve $C\subset X$ that dominates $B$ and $C\not\subset D$.
\end{conj}

\section{Reduction to $(\P^1\times B, D)$}

As a first step in our proof, we will reduce \conjref{CONJ002} to the case $(\P^1\times B,
D)$. This is also what was done in Yamanoi's proof \cite{Y}.

It is not hard to see that \eqref{E002} continues to hold after applying birational
transforms and/or base changes to $X/B$. That is, we have

\begin{lem}\label{LEM901}
Let $\pi: X\to B$ and $D$ be given as in \conjref{CONJ002}.
\begin{enumerate}
\item Let $f: X'\dashrightarrow X$ be a birational morphism and $D'$ be the proper transform of $D$ under $f$.
Then \eqref{E002} holds for $(X, D)$ if and only if it holds for $(X', D')$.
\item Let $B'\to B$ be a finite map from a smooth projective curve $B'$ to $B$, $f: X' = X\times_B B' \to X$ be the base change of the family $X$ and $D' = f^{-1}(D)$. Then
\eqref{E002} holds for $(X, D)$ if and only if it holds for $(X', D')$.
\end{enumerate}
The constants $N_\varepsilon'$ for $(X', D')$, though, might be different from $N_\varepsilon$ for $(X, D)$. 
\end{lem}

\begin{proof}
For part (1), it is enough to argue for $X'$ being the blowup of $X$ at one point $p$.
Let $C'\subset X'$ be the proper transform of $C\subset X$. Then
\begin{equation}\label{E903}
(\omega_{X/B} + D) \cdot C = (\omega_{X'/B} + D' + r E) \cdot C'
\end{equation}
for some constant $r$, where $E$ is the exceptional divisor of $f$.
On the other hand, we have
\begin{equation}\label{E904}
E\cdot C' \le X_b' \cdot C' = X_b \cdot C = \deg (C)
\end{equation}
where $X_b'$ and $X_b$ are the fibers of $X'$ and $X$ over a point $b\in B$, respectively.
Consequently,
\begin{equation}\label{E905}
|(\omega_{X/B} + D) \cdot C - (\omega_{X'/B} + D') \cdot C'| \le |r| \deg C
\end{equation}
Also it is obvious that $g(C) = g(C')$ and
\begin{equation}\label{E906}
| i_X(C, D) - i_{X'}(C', D') | \le E\cdot C' \le \deg(C)
\end{equation}
Then part (1) follows from \eqref{E905} and \eqref{E906}.

For part (2), let $d$ be the degree of the map $B'\to B$, $R\subset B'$
be its ramification locus and $\mu_r$ be the ramification index of a point $r\in R$. Let $C' = f^*(C)$.
It is not hard to see that
\begin{equation}\label{E907}
|d(\omega_{X/B} + D)\cdot C - (\omega_{X'/B'} + D')\cdot C'| \le \sum_{r\in R} (\mu_r - 1) \deg(C)
\end{equation}
\begin{equation}\label{E908}
|d(2g(C) - 2) - (2g(C') - 2)| \le \sum_{r\in R} (\mu_r - 1) \deg C
\end{equation}
and
\begin{equation}\label{E909}
|d (i_X(C, D)) - i_{X'}(C', D')| \le \sum_{r\in R} (\mu_r - 1) \deg C
\end{equation}
Then part (2) follows from \eqref{E907}-\eqref{E909}.
\end{proof}

\begin{rem}\label{REM502}
We see from the above lemma that \eqref{E501} holds regardless of the choice of birational
models $X$.
\end{rem}

\begin{rem}\label{REM901}
If $(\omega_{X/B} + D)
\cdot X_b \le 0$, \eqref{E002} is trivially true. So we may assume that
\begin{equation}\label{E927}
(\omega_{X/B} + D)\cdot X_b > 0.
\end{equation}
We may also assume that $D$ meets every fiber properly.
Using the above lemma, we can apply the stable reduction to $(X, D)$ and make $X$ into a
family of stable curves with marked points $X_b\cap D$ on each fiber.
The resulting $X$ has at worst quotient singularities and $\omega_{X/B}
+ D$ is relatively ample over $B$.
\end{rem}

\begin{prop}\label{PROP001}
If \eqref{E002} fails for some $(X, D)$, then there exists $\delta > 0$ and a log pair
$(Y, R)$ such that
\eqref{E002} fails with $(X, D, \varepsilon)$ replaced by $(Y, R, \delta)$,
where $R$ is a reduced effective divisor on $Y = \P^1\times B$.
\end{prop}

\begin{proof}
By the above remark, we may assume that $X$ is a family of stable curves with marked
points $X_b \cap D$. In particular, $\omega_{X/B} + D$ is relatively ample over $B$.

Since \eqref{E002} fails for $(X, D)$, there exists a sequence of irreducible
curves $C_1, C_2, ..., C_n, ... \subset X$ such that
\begin{equation}\label{E009}
\lim_{n\to\infty} \left(\frac{(\omega_{X/B} + D) \cdot C_n}{X_b\cdot C_n} -
\frac{(1 + \varepsilon)(2g(C_n)-2 + i_X(C_n, D))}{X_b\cdot C_n}\right) = \infty
\end{equation}

Taking a sufficiently ample line bundle $L$ on $X$, we can map $X\to \P^1$ with a very general
pencil in $|L|$. Combining this with the projection $X\to B$, we obtain a rational map
$\phi: X\dashrightarrow Y = B\times\P^1$. We can make the following happen by
taking $L$ sufficiently ample and the pencil sufficiently general:
\begin{itemize}
\item The indeterminancy locus $I_\phi$ of $\phi$ consists
of $L^2$ distinct points on $X$, $I_\phi\cap C_n = \emptyset$ for all $n$ and 
$I_\phi\cap D = \emptyset$.
\item Outside of $I_\phi$, $\phi$ is finite. Let $R_X\subset X$ be the closure of the
ramification locus of $\phi: X\backslash I_\phi\to Y$, $R_Y = \ol{\phi(R_X)}$ be the proper
transform of $R_X$ and
\begin{equation}\label{E106}
\phi^* R_Y = 2 R_X + R_\phi
\end{equation}
outside of $I_\phi$, where $R_\phi\subset X$ is a reduced effective divisor on $X$.
\item $\phi$ is simply ramified along $R_X$ with multiplicity $2$.
\item $\phi$ maps $C_n$ and $D$ birationally to $\Gamma_n = \phi(C_n)$ and
$\Delta = \phi(D)$, respectively, for all $n$.
\end{itemize}

Since $\phi_* C_n = \Gamma_n$, we have
\begin{equation}\label{E105}
\phi^* (\omega_{Y/B} + R_Y + \Delta) \cdot C_n = (\omega_{Y/B} + R_Y + \Delta) \cdot \Gamma_n
\end{equation}
On the other hand,
\begin{equation}\label{E107}
\begin{split}
&\quad \phi^* (\omega_{Y/B} + R_Y + \Delta) \cdot C_n\\
& = (\phi^* \omega_{Y/B} + 2 R_X + R_\phi +
\phi^* \Delta) \cdot C_n\\
&= (\phi^* \omega_{Y/B} + R_X + D) \cdot C_n  + (R_X + R_\phi)\cdot C_n + D_\phi \cdot C_n
\end{split}
\end{equation}
where
\begin{equation}\label{E112}
\phi^*\Delta = D + D_\phi
\end{equation}
for some effective divisor $D_\phi\subset X$. By Riemann-Hurwitz,
\begin{equation}\label{E110}
\omega_{X/B} = \phi^*\omega_{Y/B} + R_X
\end{equation}
holds outside of $I_\phi$. Meanwhile, it is obvious that
\begin{equation}\label{E108}
(R_X + R_\phi)\cdot C_n \ge i_Y(\Gamma_n, R_Y)
\end{equation}
and
\begin{equation}\label{E109}
D_\phi \cdot C_n \ge i_Y(\Gamma_n, \Delta) - i_X(C_n, D)
\end{equation}
Combining \eqref{E105} through \eqref{E109}, we obtain
\begin{equation}\label{E111}
\begin{split}
&\quad (\omega_{Y/B} + R_Y + \Delta) \cdot \Gamma_n - (1 +\delta)
\big(2g(\Gamma_n) - 2 + i_Y(\Gamma_n, R)\big)\\
& \ge 
(\omega_{X/B} + D)\cdot C_n - (1+ \delta) \big(2g(C_n) - 2 + i_X(C_n, D)\big)\\
&\quad - \delta (R_X + R_\phi + D_\phi) C_n
\end{split}
\end{equation}
where $R = R_Y\cup \Delta$.
Since $\omega_{X/B} + D$ is relatively ample over $B$,
there exist constants $\beta$ and $\gamma > 0$ such that
\begin{equation}\label{E006}
(R_X + R_\phi + D_\phi) C \le \gamma (\omega_{X/B} + D + \beta X_b) C
\end{equation}
for all curves $C\subset Y$. Thus, it suffices to take 
\begin{equation}\label{E250}
\delta = \frac{\varepsilon}{(1+\varepsilon)\gamma + 1}
\end{equation}
Then
\begin{equation}\label{E007}
\begin{split}
&\quad (\omega_{X/B} + D)\cdot C_n - (1+ \delta) \big(2g(C_n) - 2 + i_X(C_n, D)\big)\\
&\quad\quad\quad - \delta (R_X + R_\phi + D_\phi)\cdot C_n\\
&\ge  (1 - \delta \gamma)  (\omega_{X/B} + D)\cdot C_n - 
(1+ \delta) \big(2g(C_n) - 2 + i_X(C_n, D)\big) \\
&\quad\quad\quad - \beta\gamma\delta (X_b\cdot C_n)\\
&= (1 - \delta \gamma) \bigg( (\omega_{X/B} + D)\cdot C_n - (1 +
\varepsilon) \big(2g(C_n) - 2 + i_X(C_n, D)\big) \bigg)\\
&\quad\quad\quad - \beta\gamma\delta (X_b\cdot C_n)
\end{split}
\end{equation}
Therefore,
\begin{equation}\label{E008}
\lim_{n\to\infty} \left(\frac{(\omega_{Y/B} + R) \cdot \Gamma_n}{Y_b \cdot
\Gamma_n} - (1 + \delta) \frac{2 g(\Gamma_n) - 2
+ i_Y(\Gamma_n, R)}{Y_b \cdot \Gamma_n}\right) = \infty
\end{equation}
and \propref{PROP001} follows.
\end{proof}

In the above proof, we have quite an amount of freedom to choose the map $X\dashrightarrow
\P^1$. We can make $R$ really ``nice'' by choosing $L$ and the pencile of $L$ sufficiently ``general''.

\begin{prop}\label{PROP003}
Let $S$ be a finite set of points on $B$.
In the proof of \propref{PROP001}, for a sufficiently ample $L$ and a general pencil
$\sigma\subset |L|$ that maps $X\dashrightarrow \P^1$, the corresponding divisor $R = R_Y +
\Delta\subset Y = \P^1\times B$ has the
following properties:
\begin{itemize}
\item For every fiber $Y_b$ of $Y/B$,
\begin{equation}\label{E979}
i_Y(Y_b, R) \ge Y_b\cdot R - 1
\end{equation}
and if the equality holds, $b\in B\backslash S$ and $X_b$ is disjoint from the base locus
$\Bs(\sigma)$ of $\sigma$;
\item $R$ is a divisor of normal crossing.
\end{itemize}
\end{prop}

\begin{proof}
Let $\G(k, |L|)$ be the Grassmanian $\{\P^k \subset |L|\}$. For each pencil $\sigma\in \G(1,
|L|)$, we use the notation $\phi_\sigma$ for the rational map $X\dashrightarrow Y$ induced by
$\sigma$ and $R_{X,\sigma}$ for the closure of its ramification locus.
Let $\phi_{\sigma, b}: X_b\to \P^1$ be the restriction of $\phi_\sigma$ to $X_b$ and let
$R_{X,\sigma,b} = R_{X,\sigma}\cap X_b$ be the ramification locus of $\phi_{\sigma,b}$.

For $L$ sufficiently ample and for each $b\in B$, we see by simple dimension counting that all of
\begin{equation}\label{E980}
\begin{split}
\{ \sigma: &\ \phi_\sigma(p_1) = \phi_\sigma(p_2) = \phi_\sigma(p_3)\\
& \text{ for three
  distinct points } p_1,p_2,p_3\in D\cap X_b \},
\end{split}
\end{equation}
\begin{equation}\label{E981}
\begin{split}
\{ \sigma: &\ \phi_\sigma(p_1) = \phi_\sigma(p_2) = \phi_\sigma(p_3)\\
& \text{ for } p_1\ne p_2\in D\cap X_b \text{ and } p_3\in R_{X,\sigma,b}\},
\end{split}
\end{equation}
\begin{equation}\label{E804}
\begin{split}
\{ \sigma: &\ \phi_\sigma(p_1) = \phi_\sigma(p_2) \text{ and } X_b\cap \Bs(\sigma)\ne\emptyset,\\
& \text{ for } p_1\ne p_2\in D\cap X_b\}
\end{split}
\end{equation}
\begin{equation}\label{E982}
\begin{split}
\{ \sigma: &\ \phi_\sigma(p_1) = \phi_\sigma(p_2), \ \text{where}\ p_1\in D\cap X_b\
 \text{and}\\
&\phi_{\sigma,b} \text{ ramifies at } p_2\in R_{X,\sigma,b} \text{ with index $\ge 3$}\},
\end{split}
\end{equation}
\begin{equation}\label{E805}
\begin{split}
\{ \sigma: &\ \phi_\sigma(p_1) = \phi_\sigma(p_2) \text{ and } X_b\cap \Bs(\sigma)\ne\emptyset,\\
& \text{where}\ p_1\in D\cap X_b\ \text{and}\ p_2\in R_{X,\sigma,b}\},
\end{split}
\end{equation}
\begin{equation}\label{E985}
\begin{split}
\{ \sigma: &\ \phi_\sigma(p_1) = \phi_\sigma(p_2), \ \text{where}\ p_1\ne p_2\in R_{X,\sigma,b}\
 \text{and}\\
&\phi_{\sigma,b} \text{ ramifies at } p_2 \text{ with index $\ge 3$}\},
\end{split}
\end{equation}
\begin{equation}\label{E806}
\begin{split}
\{ \sigma: &\ \phi_\sigma(p_1) = \phi_\sigma(p_2) \text{ and } X_b\cap \Bs(\sigma)\ne\emptyset,\\
&\text{where}\ p_1\ne p_2\in R_{X,\sigma,b}\},
\end{split}
\end{equation}
\begin{equation}\label{E983}
\begin{split}
\{ \sigma: &\ \phi_{\sigma,b} \text{ ramifies at } p_1\ne p_2\in
R_{X,\sigma,b} \text{ with indices $\ge 3$}\}
\end{split}
\end{equation}
\begin{equation}\label{E807}
\begin{split}
\{ \sigma: &\ \phi_{\sigma,b} \text{ ramifies at } p_1\in
R_{X,\sigma,b} \text{ with index $\ge 3$}\\
&\text{ and } X_b\cap \Bs(\sigma)\ne\emptyset\} \ \text{ and}
\end{split}
\end{equation}
\begin{equation}\label{E984}
\begin{split}
\{ \sigma: &\ \phi_{\sigma,b} \text{ ramifies at } p_1\in
R_{X,\sigma,b} \text{ with index $\ge 4$}\}
\end{split}
\end{equation}
have codimension two in $\G(1, |L|)$ and hence \eqref{E979} follows. The same dimension
count also shows that $Y_b$ meets $R$ transversely for $b\in S$ and $\sigma$
general. Hence if the equality in \eqref{E979} holds, $b\not\in S$.

Already by \eqref{E979}, we see that $R$ has at worst double points as
singularities. We can further show that the singularities $R_{\sing}$ of $R$ are all
nodes.

Let $D = \sum D_i$, where $D_i$'s are irreducible components of $D$, which are sections of
$X/B$ by our assumption on $X$. And let $\Delta_{\sigma, i} = \phi_\sigma(D_i)$ and
$R_{Y,\sigma} = \phi_\sigma(R_{X,\sigma})$. To show that $R$ has normal crossing, it is
suffices to verify the following:
\begin{itemize}
\item $\Delta_{\sigma, i}$ and $\Delta_{\sigma, j}$ meet transversely for all $i\ne j$;
\item $\Delta_{\sigma, i}$ meets $R_{Y,\sigma}$ transversely for all $i$;
\item $R_{Y,\sigma}$ is nodal.
\end{itemize}
It is easy to see that the monodromy action on the intersections
$\Delta_{\sigma,i}\cap \Delta_{\sigma, j}$ when $\sigma$ varies in $\G(1, |L|)$ is
transitive. Therefore, to show that $\Delta_{\sigma, i}$ and $\Delta_{\sigma, j}$ meet
transversely, it suffices to show that they meet transversely at (at least) one point, i.e.,
\begin{itemize}
\item there exists $\sigma\in \G(1, |L|)$, $p_i\in D_i$ and $p_j\in D_j$ such that
  $\Delta_{\sigma, i}$ and $\Delta_{\sigma, j}$
  meet transversely at $\phi_\sigma(p_i) = \phi_\sigma(p_j)$.
\end{itemize}
Similarly, the other two statements translate to
\begin{itemize}
\item there exists $\sigma\in \G(1, |L|)$, $p_i\in D_i$ and $q\in R_{X,\sigma}$ such that
$\Delta_{\sigma, i}$ and $R_{Y,\sigma}$ meet transversely at $\phi_\sigma(p_i) = \phi_\sigma(q)$;
\item there exists $\sigma\in \G(1, |L|)$ and $q\in R_{X,\sigma, b}$ for some $b$ such
  that $\phi_{\sigma, b}$ has ramification index $3$ at $q$ and $R_{Y,\sigma}$ is smooth
  at $\phi_\sigma(q)$;
\item there exists $\sigma\in \G(1, |L|)$ and $q_1\ne q_2\in R_{X,\sigma, b}$ for some $b$ such
  that $R_{Y,\sigma}$ has a node at $\phi_{\sigma}(q_1) = \phi_\sigma(q_2)$.
\end{itemize}
None of these statements are hard to prove. We leave their verification to the readers.
\end{proof}

Suppose that \eqref{E002} fails for $(X, D)$ and $\{C_n\subset X\}$ is the sequence of
irreducible curves satisfying \eqref{E009}. We fix a positive $(1,1)$ form $\omega$
on $X$ that represents $c_1(L)$ and for every finite set of points $S\subset B$, we define
\begin{equation}\label{E800}
f_\omega(S) = \lim_{r\to 0} \varliminf_{n\to\infty} \left(\frac{1}{L\cdot C_n} \sum_{b\in S}
\int_{C_n\cap \pi^{-1}(U(b,r))} \omega\right)
\end{equation}
where $U(b, r)\subset B$ is the disk of radius $r$ centered at $b$. Of course, we need
a metric on $B$ in order to define $U(b, r)$. But it is obvious that the choice of metric
on $B$ is irrelevant here. Although $f_\omega(S)$ depends on
the choice of $\omega$, the vanishing of $f_\omega(S)$ does not depend on $\omega$, i.e.,
if $f_\omega(S) = 0$ for one $\omega$, it vanishes for all choices of $\omega$. And it is
easy to see that
\begin{equation}\label{E801}
\sum_\alpha f_\omega(S_\alpha) \le 1
\end{equation}
for any collection $\{S_\alpha\subset B\}$ of disjoint finite sets $S_\alpha$.

Let us fix a sufficient ample line bundle $L$ on $X$ and let $\phi_\sigma: X\dashrightarrow Y$
be the map given by a pencil $\sigma\subset |L|$ as in the proof of
\propref{PROP003}. This map gives rise to another log pair $(Y, R)$ with $R$ satisfying
the conditions given in the above proposition. Let $Q_\sigma\subset B$ be the finite set
of points $b$ where the equality in \eqref{E979} holds. This gives us a map from $\G(1,
|L|)$ to $B^N/S_N$ sending $\sigma\to Q_\sigma$, where $N = |Q_\sigma|$ and $B^N/S_N$ is
the space of $N$ unordered points on $B$. By \propref{PROP003}, $Q_\sigma\cap Q_{\sigma'}
= \emptyset$ for two general pencils $\sigma$ and $\sigma'$. Combining this with
\eqref{E801}, we see that the set $\{\sigma: f_\omega(Q_\sigma) > r\}$ is contained in a
proper subvariety of $\G(1, |L|)$ for every $r > 0$. Consequently, the set
\begin{equation}\label{E802}
\{\sigma: f_\omega(Q_\sigma) > 0\} = \bigcup_{n=1}^\infty \{\sigma: f_\omega(Q_\sigma) > \frac{1}{n}\}
\end{equation}
is contained in a union of countably many proper subvarieties of $\G(1, |L|)$. In other
words, $f_\omega(Q_\sigma) = 0$ for a very general pencil $\sigma$. For a very general
pencil $\sigma$, $C_n$ are disjoint from the base locus of $\sigma$. 
Hence $L\cdot C_n = Y_p \cdot \Gamma_n$, where $\Gamma_n = \phi_\sigma(C_n)$ and $Y_p$ is
a fiber of $Y/\P^1$. In addition, we have proved that $X_b\cap \Bs(\sigma) = \emptyset$
for $b\in Q_\sigma$. Hence $f_\omega(Q_\sigma) = 0$ implies
\begin{equation}\label{E803}
\lim_{r\to 0} \varliminf_{n\to\infty} \left(\frac{1}{Y_p\cdot \Gamma_n} \sum_{b\in Q_\sigma}
\int_{\Gamma_n\cap \pi_Y^{-1}(U(b,r))} \eta\right) = 0
\end{equation}
where $\eta$ is the pullback of a positive $(1,1)$ form on $\P^1$ representing
 $c_1(\CO_{\P^1}(1))$ and $\pi_Y$ is the projection $Y\to B$.
By taking a
subsequence of $\{\Gamma_n\}$, we may as well replace $\varliminf$ by $\lim$.

We may further apply a suitable base change to $Y/B$ to make $R_Y$ into a union of
sections of $Y/B$ while preserving the other properties of $(Y, R)$. 
So we finally reduce the conjecture from $(X, D, \varepsilon)$ to $(Y, R,
\delta)$. Replacing $(X, D, \varepsilon)$ by $(Y, R, \delta)$, we may assume the following holds.
\begin{itemize}
\item[{\bf A1.}] $D\subset X = \P^1\times B$ is a normal-crossing divisor which is a union of sections of
$X/B$.
\item[{\bf A2.}] $\omega_{X/B} + D$ is relatively ample over $B$.
\item[{\bf A3.}] For every fiber $X_b$ of $X/B$,
\begin{equation}\label{E919}
i_X(X_b, D) \ge X_b\cdot D - 1
\end{equation}
\item[{\bf A4.}] There is a sequence of reduced and irreducible curves $\{C_n\}$ on $X$
  that dominate $B$ and satisfy \eqref{E009}.
\item[{\bf A5.}] Let $Q\subset B$ be the set of points $b$ where the equality in
  \eqref{E919} holds, i.e., $Q = \pi(D_{\sing})$, where $D_{\sing}$ is the singular locus
  $D_{\sing}$ of $D$; then
\begin{equation}\label{E808}
\lim_{r\to 0} \lim_{n\to\infty} \left(\frac{1}{X_p\cdot C_n} \sum_{b\in Q}
\int_{C_n\cap \pi^{-1}(U(b,r))} w\right) = 0
\end{equation}
where $X_p$ is the fiber of $X$ over a point $p\in \P^1$ and $w$ is the pullback of
a positive $(1,1)$ form on $\P^1$ representing $c_1(\CO_{\P^1}(1))$.
\end{itemize}

\section{Proof of \conjref{CONJ002}}

\subsection{Lifts to the first jet space}

Now we can work exclusively on $(X,D)$ with $(X, D)$ satisfies the hypotheses A1-A5 in the
last section.
As in Vojta's proof of $2+\varepsilon$ theorem, we start by lifting every curve $C_n\subset
X$ to its 1st jet space.

Let $\Omega_X(\log D)$ be the sheaf of
logarithmic differentials with poles along $D$ and $T_X(-\log D) = \Omega_X(\log D)^\vee$
be its dual.
Let $Y = \P T_X(-\log D)$ be the
projectivization of $T_X(-\log D)$ with tautological line bundle $L$.
Here we follow the traditional convention that
\begin{equation}\label{E920}
\P E = \Proj (\oplus \Sym^\bullet E^\vee)
\text{ and }
H^0(L)\isom H^0(E^\vee).
\end{equation}

We have the basic exact sequence
\begin{equation}\label{E014}
0\xrightarrow{} \pi^* \Omega_B \xrightarrow{} \Omega_X(\log D) \xrightarrow{}
\Omega_{X/B}(D)
\end{equation}
Note that this sequence is not right exact; $\Omega_X(\log D)\to \Omega_{X/B}(D)$
fails to be surjective along $D_{\sing}$.

Every nonconstant map $\nu: C\to X$ from a smooth curve $C$ to $X$
can be naturally lifted to a map $\nu_Y: C\to Y$ via the map
\begin{equation}\label{E015}
\P T_C(-\log \nu^* D) \to \P T_X(-\log D)
\end{equation}
Suppose that $\nu$ maps $C$ birational onto its image.
Then the natural map $\nu^* \Omega_X (\log D) \to \Omega_C(\log \nu^* D)$ induces a map
\begin{equation}\label{E017}
\nu_Y^* L \rightarrow \Omega_C(\log \nu^* D)
\end{equation}
Obviously, this map is nonzero and $\nu_Y^* L$ is locally free; consequently, it is an injection.
Therefore, we have
\begin{equation}\label{E010}
\deg \nu_Y^* L \le \deg \Omega_C(\log \nu^* D) = 2g(C) - 2 + i_X(\nu(C), D)
\end{equation}
Hence \eqref{E002} holds if
\begin{equation}\label{E012}
\deg \nu_Y^* (\pi_X^*(\omega_{X/B} + D) - (1+\varepsilon) L) \le N_\varepsilon \deg (\nu^* X_b)
\end{equation}
where $\pi_X$ is the projection $Y\to X$. Another way to put this is that
\begin{equation}\label{E013}
G \cdot (\nu_Y)_* C \ge 0
\end{equation}
for a sufficiently ample divisor $M\subset B$ and every $\nu: C\to X$ with $\nu(C)$ dominating $B$,
where
\begin{equation}\label{E200}
G = (1+\varepsilon) L + \pi_B^* M - \pi_X^*(\omega_{X/B} + D)
\end{equation}
where $\pi_B = \pi\circ \pi_X$ is the projection $Y\to B$. Or in the context of our
hypothesis A4, we want to show that
\begin{equation}\label{E809}
-G\cdot \Gamma_n = O(\deg C_n)
\end{equation}
and thus arrive at a contradiction,
where $\Gamma_n\subset Y$ is the lift of $C_n\subset X$ via its normalization and $\deg
C_n = C_n \cdot X_b$.
Here by $O(\deg C_n)$, we mean a quantity $\le K \deg C_n$ for some constant $K$ and all $n$.

Obviously, \eqref{E013} holds if the divisor $G$ is numerically effective (NEF).
Unfortunately, we cannot expect this to be true in general.

The map $\Omega_X(\log D)\to \Omega_{X/B}(D)$ in \eqref{E014} induces a rational map
\begin{equation}\label{E018}
\P T_{X/B}(-D) \dashrightarrow Y.
\end{equation}
Let $\Delta\subset Y$ be the closure of the image of this map.
As we are going to see, $\Delta$ will play a central role in our
argument. Another way to characterize $\Delta$ is the following.

\begin{lem}\label{LEM902}
We have
\begin{equation}\label{E213}
\Delta = \overline{\bigcup_{b\in B} \mu_Y(X_b)}
\end{equation}
and a curve $\nu: C\hookrightarrow X$ is tangent to a fiber $X_b$ if and only if
$\nu_Y(C)$ intersects $\Delta$, where $\mu_Y: X_b \to Y$ is the lifting of the embedding
$X_b\hookrightarrow X$.
\end{lem}

\begin{proof}
This is more or less trivial.
\end{proof}

\subsection{Some Numerical Results}

Here we prove some numerical results on $\Delta, L$ and $G$, which we are going to need later. First of
all, it is obvious that $\pi_X$ maps $\Delta$ birationally onto $X$; indeed, by a
local analysis, we see that $\Delta$ is the blowup of $X$ along $D_{\sing}$, i.e.,
the places where $\Omega_X(\log D)\to \Omega_{X/B}(D)$ failes to be surjective.
In the lift of $\nu: C\to X$ to $\nu_Y: C\to Y$, if $\nu$ is a smooth embedding, we have
\begin{equation}\label{E911}
(\nu_Y)^* L = \omega_C + \nu^{-1}(D)
\end{equation}
where $\nu^{-1}(D) = \supp(\nu^* D)$ is the reduced pre-image of $D$.
Namely, \eqref{E017} is an isomorphism. Therefore, for every fiber
$X_b$,
\begin{equation}\label{E912}
L\cdot \wt{X}_b = 2g(X_b) - 2 + i_X(X_b, D)
\end{equation}
where $\wt{X}_b\subset \Delta$ is the proper transform of $X_b$ under $\Delta\to X$.
Applying this to all the fibers $X_b$ with $X_b\cap D_{\sing}\ne \emptyset$, we see that
\begin{equation}\label{E913}
L \big|_\Delta = \pi_X^*(\omega_{X/B} + D + \pi^* M) - E
\end{equation}
for some divisor $M$ on $B$, where
\begin{equation}\label{E215}
E = \sum_{q\in D_{\sing}} E_q
\end{equation}
is the exceptional divisor of $\Delta\to X$. To determine $M$, we restrict everything to a section
$X_p = \rho^{-1}(p)$ of $X/B$, where $\rho$ is the projection $X\to \P^1$.
For $p$ general, the restriction of \eqref{E014} to $X_p\isom B$ becomes
\begin{equation}\label{E917}
0\xrightarrow{} \Omega_{X_p} \xrightarrow{} \Omega_X(\log D)\big|_{X_p} \xrightarrow{}
\CO_{X_p}(D) \xrightarrow{} 0
\end{equation}
Let $\Delta_p$ be the proper transform of $X_p$ under $\Delta\to X$. Then we see from
\eqref{E917} that the restriction of $L$ to $\Delta_p\isom B$ is 
\begin{equation}\label{E914}
L\big|_{\Delta_p} = \pi_X^* D
\end{equation}
Comparing \eqref{E913} and \eqref{E914}, we conclude that $M$ is trivial and hence
\begin{equation}\label{E915}
L \big|_\Delta = \pi_X^*(\omega_{X/B} + D) - E
\end{equation}
As a consequence,
\begin{equation}\label{E011}
\begin{split}
G\Big|_\Delta &= \big((1+\varepsilon) L + \pi_B^* M - \pi_X^*(\omega_{X/B} + D)\big)
  \Big|_\Delta\\
&= \varepsilon \pi_X^*(\omega_{X/B} + D) + \pi_B^* M - (1+\varepsilon) E
\end{split}
\end{equation}
Next, we claim that
\begin{equation}\label{E918}
\Delta = L - \pi_B^* \omega_B
\end{equation}
This is obviously true if \eqref{E014} is an exact sequence of locally free sheaves, i.e.,
$D_{\sing} = \emptyset$. To see this is true in general, we restrict everything to a
smooth curve $C\subset X$ with $C\cap D_{\sing} = \emptyset$. By the above reason, \eqref{E918}
holds when restricted to $\pi_X^{-1}(C)$. Such curves $C$ obviously generate
$\Pic(X)$ and hence \eqref{E918} holds over $Y$.

%
By restricting \eqref{E014} to each fiber $X_b$ of $X/B$, we see that $L$ is relatively
NEF over $B$. Moreover, the following holds.

\begin{lem}\label{LEM904}
For all $m \ge k\in \BZ$ and $m > 0$, $mL - k\Delta$ is relatively BPF over $B$ and
\begin{equation}\label{E950}
H^1(m (L + \pi_B^* M) - k\Delta) = 0
\end{equation}
for a sufficiently ample divisor $M\subset B$.
\end{lem}

\begin{proof}
Since $c_1(\Omega_X(\log D)) = \omega_X + D$, the restriction of $\Omega_X(\log D)$ to
a fiber $X_b\isom \P^1$ is
\begin{equation}\label{E951}
\Omega_X(\log D)\big|_{X_b} = \CO_{\P^1}(\beta) \oplus \CO_{\P^1}(\gamma)
\end{equation}
where
\begin{equation}\label{E952}
\beta + \gamma = \alpha = (\omega_{X/B} + D)\cdot X_b.
\end{equation}
By \eqref{E014}, we must have $\beta,\gamma\ge 0$. Therefore,
\begin{equation}\label{E953}
Y_b \isom \P\left(\CO_{\P^1}(-\beta) \oplus \CO_{\P^1}(-\gamma)\right)
\end{equation}
and together with \eqref{E918}, we see that $mL - k\Delta$ is relatively NEF over $B$ for
$m\ge k$. Also we see from the above argument that
\begin{equation}\label{E954}
H^1(Y_b, mL - k\Delta) = 0 \Leftrightarrow R^1 (\pi_B)_* \CO(mL - k\Delta) = 0
\end{equation}
This implies
\begin{equation}\label{E955}
\begin{split}
H^1(m (L + \pi_B^* M) - k\Delta) &= H^1((\pi_B)_* \CO(m (L + \pi_B^* M) - k\Delta))\\
&= H^1((\pi_B)_* L^{m-k} \tensor \CO_B(k \omega_B + m M))
\end{split}
\end{equation}
By \eqref{E953},
\begin{equation}\label{E957}
\Sym^n H^0(Y_b, L) = H^0(Y_b, L^n).
\end{equation}
Therefore,
\begin{equation}\label{E956}
H^1(m (L + \pi_B^* M) - k\Delta) 
= H^1(\Sym^{m-k} (\pi_B)_* L \tensor \CO_B(k \omega_B + m M))
\end{equation}
It suffices to choose $M$ such that all of $M$, $\omega_B + M$ and $(\pi_B)_* L\tensor \CO_B(M)$
are ample and \eqref{E950} follows.
\end{proof}

\begin{rem}\label{REM902}
It is possible to give a more precise version of \eqref{E950} on how ample $M$ should be
in terms of $\omega_B$ and $D$; however, we have no need for it here. Also in the above
proof, we observe that $L$ fails to be ample on $Y_b$ if and only if \eqref{E951} splits
as
\begin{equation}\label{E958}
\Omega_X(\log D)\big|_{X_b} = \CO_{\P^1} \oplus \CO_{\P^1}(\alpha)
\end{equation}
If \eqref{E958} holds on a general fiber $X_b$, it holds everywhere and
this only happens when $D$ consists of $\alpha+2$ disjoint sections of $X/B$, in which
case the conjecture is trivial. Hence we may assume that $L$ is ample on a general fiber
of $Y/B$. This implies that $L+ \pi_B^* M$ is big for a sufficiently ample divisor $M\subset
B$, in addition to being NEF as already proved.
The same, of course, holds for $mL - k\Delta +  \pi_B^* M$ when $m > k$.
\end{rem}

\subsection{Bergman Metric}

Given a line bundle $L$ on a compact complex manifold $X$ and sections $s_0, s_1, ..., s_n\in
|L|$ of $L$, we recall the {\it Bergman metric\/} associated to $\{s_k\}$ is the pullback of
the Fubini-Study metric under the map $X\dashrightarrow \P^n$ given by $\{s_k\}$, i.e., the pseudo-metric
with associated $(1,1)$ form
\begin{equation}\label{E024}
w = \frac{\sqrt{-1}}{2\pi} \partial \ol\partial \log \left(\sum_{k=0}^n |s_k|^2\right)
\end{equation}
Alternatively, Fubini-Study metric can be regarded as a metric of the line bundle
$\CO_{\P^n}(1)$ and the Bergman metric is correspondingly a pseudo-metric of $L$ with $w$
the curvature form.
In general, $w$ is only a closed real current of type $(1,1)$ with the following
properties:
\begin{itemize}
\item it is $C^\infty$ outside of the base locus $\Bs\{s_k\}$
of $\{s_k\}$;
\item it represents $c_1(L)$ if $\{s_k\}$ is BPF;
\item we always have
\begin{equation}\label{E025}
\nu^*w\ \text{is}\ C^\infty, \nu^*w \ge 0 \text{ and } \deg(\nu^* L) \ge \int_{C} \nu^* w
\end{equation}
for any morphism $\nu: C\to X$ from a smooth and irreducible projective curve $C$ to $X$ with 
$\nu(C) \not\subset \Bs\{s_k\}$.
\end{itemize}

The indeterminancy of
the rational map $\phi: X\dashrightarrow \P^n$ given by $\{s_k\}$ can be resolved by a
sequence of blowups along smooth centers over $\Bs\{s_k\}$. That is, there exists a birational map $\pi:
Y\to X$ such that $f = \phi\circ \pi$ is regular. Let $\wt{s}_k$ be the proper transform
of $s_k$ under $\pi$. Then $\{\wt{s}_k\}$ span a BPF linear system of $\wt{L} =
f^*\CO_{\P^n}(1)$. Let $\wt{w}$ be the Bergman metric associated to
$\{\wt{s}_k\}$. Then $\wt{w} = \pi^* w$ outside of exceptional locus of $\pi$. Indeed, the
current $w$ is defined in the way of
\begin{equation}\label{E959}
\langle w, \gamma\rangle = \int_Y \wt{w}\wedge \pi^*\gamma
\end{equation}
Then \eqref{E025} follows easily.
 
\subsection{Construction of the 1st chern classes}

Let
\begin{equation}\label{E259}
\pi_X^* (\omega_{X/B} + D) = \alpha Y_p + \pi_B^* N
\end{equation}
for some divisor $N\subset B$, where $Y_p$ is a fiber of $Y/\P^1$. 
We replace $M$ by $M + N$ and write $G$ in the form
\begin{equation}\label{E260}
G = (1+\varepsilon) L + \pi_B^* M - \alpha Y_p
\end{equation}
Our purpose remains, of course, to show \eqref{E809}.

We write the LHS of \eqref{E809} in the integral form:
\begin{equation}\label{E021}
G\cdot \Gamma_n = \int_{\Gamma_n} c_1(G) = \int_{\Gamma_n\backslash U} c_1(G) +
\int_{\Gamma_n\cap U} c_1(G)
\end{equation}
where $U$ is an (analytic) open neighborhood of $\Delta$. Here we have to work with the
forms that represent the first chern classes instead of cohomology classes themselves,
i.e., $c_1(G)$ refers to a $(1,1)$ form representing the first chern class of $G$;
otherwise, the integrals in \eqref{E021} do not make sense. The construction of
appropriate $c_1(G)$ is one of the main parts of our proof. Basically, by a proper choice
of $c_1(G)$ with
\begin{equation}\label{E023}
c_1(G) = c_1((1+\varepsilon) L + \pi_B^* M) - c_1(\alpha Y_p)
\end{equation}
we will show that both
\begin{equation}\label{E022}
-\int_{\Gamma_n\backslash U} c_1(G)\ \text{and}\
-\int_{\Gamma_n\cap U} c_1(G)
\end{equation}
have order of $O(\deg C_n)$. The forms representing $c_1((1+\varepsilon) L + \pi_B^* M)$
and $c_1(\alpha Y_p)$ are constructed via Bergman metric mentioned above.

Let us first fix a sufficiently large integer $m$ with $m\varepsilon \in \BZ$; obviously, we
may assume $\varepsilon\in \BQ$. Since $H^0(m\alpha Y_p) = H^0(\CO_{\P^1}(m\alpha))$, a
general pencil of $m\alpha Y_p$ is BPF.
To construct a form $w$ representing $c_1(m\alpha Y_p)$, it is enough to choose a BPF
pencil of $m\alpha Y_p$ with basis $\{s_0, s_1\}$ and let
\begin{equation}\label{E026}
w = \frac{\sqrt{-1}}{2\pi} \partial \ol\partial \log \left(|s_0|^2 + |s_1|^2\right)
\end{equation}
be the Bergman metric associated to $\{s_0, s_1\}$. Obviously, $w$ is $C^\infty$ and
represents $c_1(m\alpha Y_p)$. Next we will construct a Bergman metric on the line bundle
$\CO_Y(m (1+\varepsilon) L + m \pi_B^* M)$.

Let $S_i = \{s_i = 0\}$ for $i=0,1$ and let $\{\sigma_{0j}: j\in J\}$ be a basis of the linear system of
$m (1+\varepsilon) L + m \pi_B^* M$ consisting of sections $\sigma$ with
\begin{equation}\label{E264}
\sigma\Big|_{S_0} \in H^0(S_0, m (1+\varepsilon) L + m \pi_B^* M - 2\Delta)
\end{equation}
Or equivalently, $\sigma_{0j}$ are the sections tangent to $S_0$ along $S_0\cap \Delta$.

\begin{lem}\label{LEM003}
For each $j$, there exists a section $\sigma_{1j}$ of $m (1+\varepsilon) L + m \pi_B^* M$
such that $s_0 \sigma_{1j} - s_1 \sigma_{0j}$ vanishes to the order of $2$ along $\Delta$,
i.e.,
\begin{equation}\label{E027}
s_0 \sigma_{1j} - s_1 \sigma_{0j} \in H^0(m (1+\varepsilon) L + m \pi_B^* M + m\alpha Y_p - 2\Delta)
\end{equation}
In addition, $\{\sigma_{1j}\}$ can be chosen to be a basis of the linear system consisting
of sections $\sigma$ with
\begin{equation}\label{E265}
\sigma\Big|_{S_1} \in H^0(S_1, m (1+\varepsilon) L + m \pi_B^* M - 2\Delta)
\end{equation}
\end{lem}

\begin{proof}
Let $F_0$ be the subscheme of $Y$ given by $F_0 = S_0\cap 2\Delta$. Then we have the
Koszul complex for the ideal sheaf $I_{F_0}$ of $F_0\subset Y$: 
\begin{equation}\label{E266}
0\xrightarrow{} \CO(-S_0 - 2\Delta) \xrightarrow{} \CO(-S_0)\oplus \CO(-2\Delta)
\xrightarrow{} I_{F_0} \xrightarrow{} 0
\end{equation}
Obviously, 
\begin{equation}\label{E272}
\Sigma_0 = H^0(\CO_Y(m (1+\varepsilon) L + m \pi_B^* M)\tensor I_{F_0})
\end{equation}
is exactly the linear system $\Span\{\sigma_{0j}\}$ generated by $\{\sigma_{0j}\}$. By \lemref{LEM904},
\begin{equation}\label{E268}
\begin{split}
&\quad H^1(m (1+\varepsilon) L + m \pi_B^* M + m\alpha Y_p - S_0 - 2\Delta)\\
&= H^1(m (1+\varepsilon) L + m \pi_B^* M - 2\Delta) = 0
\end{split}
\end{equation}
Therefore, $AF+BG$ holds for
\begin{equation}\label{E267}
s_1\sigma_{0j} \in H^0(\CO_Y(m (1+\varepsilon) L + m \pi_B^* M + m\alpha Y_p)\tensor I_{F_0})
\end{equation}
That is,
\begin{equation}\label{E269}
s_1\sigma_{0j} = s_0 \sigma_{1j} + s_\Delta^2 l_j
\end{equation}
for some $\sigma_{1j}$, where $\Delta = \{s_\Delta = 0\}$ and $l_j$ is a section of
\begin{equation}\label{E270}
\CO_Y(m (1+\varepsilon) L + m \pi_B^* M + m\alpha Y_p - 2\Delta)
\end{equation}
And \eqref{E027} follows.
Obviously, $\sigma_{1j}$ are members of the linear sytem
\begin{equation}\label{E271}
\Sigma_1 = H^0(\CO_Y(m (1+\varepsilon) L + m \pi_B^* M)\tensor I_{F_1})
\end{equation}
where $I_{F_1}$ is the ideal sheaf of the subscheme $F_1 = S_1\cap 2\Delta\subset Y$. 
It is obvious that
\begin{equation}\label{E273}
H^0(m (1+\varepsilon) L + m \pi_B^* M - 2\Delta) \subset \Sigma_0\cap \Sigma_1
\end{equation}
It is not hard to see that $\{\sigma_{1j}\}$ spans the quotient
\begin{equation}\label{E274}
\Sigma_1/H^0(m (1+\varepsilon) L + m \pi_B^* M - 2\Delta) = \Span\{\sigma_{1j}\}
\end{equation}
Without loss of generality, we may assume that $\{\sigma_{0j}: j\in J\}$ contains a subset
$\{\sigma_{0j}: j\in J_\Delta\}$ which is a basis of $H^0(m (1+\varepsilon) L + m \pi_B^*
M - 2\Delta)$, where $J_\Delta\subset J$. Then it is enough to choose $\sigma_{1j} =
\sigma_{0j}$ for $j\in J_\Delta$. Combining this with \eqref{E274}, we see that
$\{\sigma_{1j}\}$ is a basis of $\Sigma_1$.
\end{proof}

Let $\sigma_{1j}$ be the sections given in the above lemma. Together with $\{\sigma_{0j}\}$
we have the Bergman metric associated to $\{\sigma_{ij}: 0\le i\le 1, j\in J\}$
\begin{equation}\label{E262}
\gamma = \frac{\sqrt{-1}}{2\pi} \partial \ol\partial \log \left(\sum_{ij}
|\sigma_{ij}|^2\right)
\end{equation}
And we let 
\begin{equation}\label{E263}
\eta = \gamma - w
\end{equation}

\begin{prop}\label{PROP002}
Let $\Sigma_i$ be the linear system generated by $\{\sigma_{ij}: j\in J\}$ as in
\eqref{E272} and \eqref{E271}. For each $i$, the base locus of $\Sigma_i$ is contained in
$(Y_Q \cup S_i)\cap \Delta$, where $Q = \pi(D_{\sing})\subset B$ is the finite set defined in A5
and $Y_Q = \pi_B^{-1}(Q)$.
\end{prop}

\begin{proof}
Since $H^0(m (1+\varepsilon) L + m \pi_B^* M - 2\Delta) \subset \Sigma_i$, the base locus
$\Bs(\Sigma_i)$ of $\Sigma_i$ is contained in $\Delta$ by \lemref{LEM904}. 
So it suffices to show that $\Bs(\Sigma_i)\subset Y_Q\cup S_i$.

Let $F_i =
S_i\cap 2\Delta$ be the subscheme of $Y$ defined in the proof of \lemref{LEM003}. We have
the exact sequence
\begin{equation}\label{E815}
\begin{split}
0 &\xrightarrow{} \CO_Y(m (1+\varepsilon) L + m \pi_B^* M - 2\Delta) \\
& \xrightarrow{}
\CO_Y(m (1+\varepsilon) L + m \pi_B^* M)\tensor I_{F_i}\\
& \xrightarrow{}
\underbrace{\CO_\Delta(m (1+\varepsilon) L + m \pi_B^* M - S_i)}_{\CO_\Delta(mG)} \tensor \CO_Y/I_\Delta^2 
\xrightarrow{} 0
\end{split}
\end{equation}
where $I_\Delta = \CO_Y(-\Delta)$ is the ideal sheaf of $\Delta\subset Y$. Again by
\lemref{LEM904}, $H^1(\CO_Y(m (1+\varepsilon) L + m \pi_B^* M - 2\Delta)) = 0$ and hence
we have the surjection
\begin{equation}\label{E290}
\begin{split}
\Sigma_i & \twoheadrightarrow H^0(\CO_\Delta(m (1+\varepsilon) L + m \pi_B^* M - S_i)
\tensor \CO_Y/I_\Delta^2)\\
&\quad \isom H^0(\CO_\Delta(mG) \tensor \CO_Y/I_\Delta^2)
\end{split} 
\end{equation}
Composing the above map with
\begin{equation}\label{E816}
\begin{split}
\varphi: H^0(\CO_\Delta(mG) \tensor \CO_Y/I_\Delta^2) &\to H^0(\CO_\Delta(mG)\tensor
\CO_Y/I_\Delta)\\
&\quad = H^0(\CO_\Delta(mG))
\end{split}
\end{equation}
we have a natural map 
\begin{equation}\label{E818}
f: \Sigma_i \to H^0(\CO_\Delta(mG)).
\end{equation}
To show that $\Bs(\Sigma_i)\subset Y_Q \cup S_i$, it is
enough to show that
\begin{equation}\label{E817}
\Bs(f(\Sigma_i)) \subset Y_Q
\end{equation}
which is equivalent to
\begin{equation}\label{E291}
\Bs(\IM(\varphi)) \subset Y_Q
\end{equation}
by \eqref{E290}. For $M\subset B$ sufficiently ample, we have the diagram
\begin{equation}\label{E819}
\xymatrix{
H^0(\CO_\Delta(mG) \tensor \CO_Y/I_\Delta^2)\ar[d]^\varphi \ar@{>>}[r] & H^0(\Delta_b, \CO_\Delta(mG)
\tensor \CO_Y/I_\Delta^2) \ar[d]^{\varphi_b}\\
H^0(\CO_\Delta(mG)) \ar@{>>}[r] & H^0(\Delta_b, \CO_\Delta(mG))}
\end{equation}
with rows being surjections when we restrict $\varphi$ to each fiber $\Delta_b$ of $\Delta/B$ for $b\not\in Q$.
Therefore, it suffices to show that
\begin{equation}\label{E820}
\Bs(\IM(\varphi_b)) = \emptyset
\end{equation}
for all $b\not\in Q$. This is more or less obvious since
we have the exact sequence
\begin{equation}\label{E941}
\xymatrix{0 \ar[r] & I_\Delta /I_\Delta^2 \ar[r]\ar@{=}[d] & \CO_Y/I_\Delta^2 \ar[r]
& \CO_Y/I_\Delta \ar[r] \ar@{=}[d] & 0\\
& \CO_\Delta(-\Delta) & & \CO_\Delta}
\end{equation}
When we tensor the sequence by $\CO_\Delta(mG)$ and restrict it to $\Delta_b\isom \P^1$ with
$b\not\in Q$, we have
\begin{equation}\label{E821}
h^1(\Delta_b, \CO_\Delta(mG - \Delta)) = h^1(\CO_{\P^1}((m\varepsilon - 1)\alpha)) = 0 
\end{equation}
by \eqref{E011} and \eqref{E918}. Consequently, $\varphi_b$ is surjective and
\begin{equation}\label{E942}
\Bs(\IM(\varphi_b)) = \Bs(H^0(\Delta_b, \CO_\Delta(mG))) =
\Bs(H^0(\CO_{\P^1}(m\varepsilon\alpha))) = \emptyset
\end{equation}
\end{proof}

\begin{rem}\label{REM001}
It is not hard to see that the above proposition continues to hold
with tangency $2$ replaced by any $\mu \le m\varepsilon$. Moreover,
being a little more careful, we can actually show that
\begin{equation}\label{E822}
\Bs(\Sigma_i) = \wt{X}_Q \cup (S_i\cap \Delta)
\end{equation}
where $\wt{X}_Q\subset \Delta$ is the proper transform of $X_Q = \pi^{-1}(Q)$ under the
map $\Delta\to X$. However, we have no need for these here.
\end{rem}

By the above proposition, we see that the base locus of
$\{\sigma_{ij}: i,j\}$ is supported on $Y_Q\cap \Delta$.
Consequently, $\gamma$ is a closed $(1,1)$ current which is $C^\infty$ on $Y\backslash
(Y_Q\cap \Delta)$.
By \eqref{E025},
\begin{equation}\label{E028}
\begin{split}
-mG\cdot \Gamma_n &\le -\int_{\Gamma_n} \eta = -\int_{\Gamma_n \backslash U} \eta
- \int_{\Gamma_n \cap U} \eta\\
&\le \int_{\Gamma_n \backslash U} w - \int_{\Gamma_n\cap U} \eta
\end{split}
\end{equation}
The fact that the first integral has order of $O(\deg C_n)$ is a consequence of the
following lemma.

\begin{lem}\label{LEM903}
Let $U\subset Y$ be an open neighborhood of $\Delta$, $w$ be a smooth $(1,1)$ form
on $X$ and $\kappa$ be a positive smooth $(1,1)$ form on $B$.
Then there exists a constant $A_U > 0$ such that at every point $(p, v) \in Y\backslash U$
\begin{equation}\label{E277}
\big|\langle w, v\wedge \ol{v}\rangle\big|
\le A_U \langle \pi^* \kappa, v\wedge \ol{v}\rangle
\end{equation}
where $p\in X$ and $v\in T_{X,p}(-\log D)$.
\end{lem}

\begin{proof}
By \lemref{LEM902}, $\langle \pi^* \kappa, v\wedge \ol{v}\rangle$ does not vanish for
$(p,v)\not\in \Delta$ and hence the function
\begin{equation}\label{E278}
f(p, v) = \frac{\langle w, v\wedge \ol{v}\rangle}{\langle \pi^* \kappa, v\wedge
  \ol{v}\rangle}
\end{equation}
is continuous on $Y\backslash \Delta$. Then \eqref{E277} follows from the compactness of
$Y\backslash U$.
\end{proof}

Note that $w$ is the pullback of a form on $X$; indeed, it is the pullback of a form on $\P^1$.
So \lemref{LEM903} applies and we conclude that
\begin{equation}\label{E029}
w \le A_U \pi_B^* \kappa
\end{equation}
on $\Gamma_n \backslash U$ for some constant $A_U$ depending only on $U$, where we
choose $\kappa$ to be a positive $(1,1)$ form on $B$ representing $c_1(\CO_B(b))$ for a point $b\in B$. Therefore,
\begin{equation}\label{E279}
\int_{\Gamma_n \backslash U}  w \le A_U \int_{\Gamma_n} \pi_B^*\kappa = A_U \deg(C_n)
\end{equation}
Next, we claim that $\eta > 0$ everywhere on $\Delta\backslash Y_Q$.

\begin{lem}\label{LEM001}
The current $\eta > 0$ at every point $p\in \Delta\backslash Y_Q$.
\end{lem}

By \eqref{E808}, there exists an open neighborhood $V$ of $Y_Q$ such that
\begin{equation}\label{E811}
\int_{\Gamma_n \cap V} w \le \varepsilon (m\alpha Y_p \cdot \Gamma_n)
\end{equation}
By the above lemma and the compactness of $\Delta\backslash V$,
we see that $\eta > 0$ in $U\backslash V$ for some open neighborhood $U$ of
$\Delta$. The second integral in \eqref{E028} becomes
\begin{equation}\label{E810}
\begin{split}
- \int_{\Gamma_n\cap U} \eta & \le - \int_{\Gamma_n\cap (U\backslash V)} \eta +
  \int_{\Gamma_n\cap V} w\\
& \le \varepsilon (m\alpha Y_p \cdot \Gamma_n) = m\varepsilon
  (\omega_{X/B} +D)\cdot C_n + O(\deg C_n)
\end{split}
\end{equation}
Combining \eqref{E279} and \eqref{E810}, we have
\begin{equation}\label{E812}
\begin{split}
&\quad -G\cdot \Gamma_n = \varepsilon (\omega_{X/B} +D)\cdot C_n + O(\deg C_n)\\
&\Rightarrow
-\big((1+\varepsilon) L + \pi_B^* M - (1-\varepsilon) \pi_X^*(\omega_{X/B} + D)\big)\cdot
\Gamma_n = O(\deg C_n)
\end{split}
\end{equation}
Replace $\varepsilon$ by $\varepsilon/(2+\varepsilon)$ and we are done.
It remains to verify \lemref{LEM001}.

\begin{proof}[Proof of \lemref{LEM001}]
At least one of $s_0(p)$ and $s_1(p)$ does not vanish.  
Let us assume that $s_0(p) \ne 0$ WLOG. Let $r_j = \sigma_{0j}/s_0$; $r_j$ is holomorphic
at $p$, of course. Let $\delta_j = \sigma_{1j} - s_1 r_j$. By our construction of
$\sigma_{1j}$, we see that $\delta_j$ vanishes to the order of $2$ along $\Delta$. We may
write
\begin{equation}\label{E280}
\begin{split}
\gamma &= \frac{\sqrt{-1}}{2\pi} \partial \ol\partial \log\left(\sum_{j} (|s_0 r_j|^2 + |s_1
r_j + \delta_j|^2) \right)\\
& = \underbrace{\frac{\sqrt{-1}}{2\pi} \partial \ol\partial \log (|s_0|^2 + |s_1|^2)}_w +
\frac{\sqrt{-1}}{2\pi} \partial \ol\partial \log \left(\sum_j |r_j|^2\right)\\
&\quad - \frac{\sqrt{-1}}{2\pi} \partial \ol\partial \log \left(1 + \sum_j \frac{s_1 r_j
  \ol{\delta}_j + \ol{s}_1 \ol{r}_j \delta_j + |\delta_j|^2}{(|s_0|^2 + |s_1|^2)\sum_j |r_j|^2}\right)
\end{split}
\end{equation}
Basically, we want to show that the last term in \eqref{E280} vanishes along
$\Delta$. Then
\begin{equation}\label{E281}
\eta\Big|_\Delta = \frac{\sqrt{-1}}{2\pi} \partial \ol\partial \log \left(\sum_j |r_j|^2\right)
\end{equation}
locally at $p$, which is positive.

Since $\eta$ is $C^\infty$ at $p$, it is enough to show that $\eta > 0$ at $p$ when $\eta$ is restricted
to every curve passing through $p$, i.e., to show that $f^* \eta > 0$ at $q$ for every
nonconstant morphism $f: C\to Y$ from a smooth and irreducible projective $C$ to $Y$ with
$f(q) = p$. Indeed, it is enough to show the following
\begin{quote}
for every tangent vector $\xi\in T_{Y, p}$, there exists a morphism $f:C\to Y$ from a smooth
irreducible curve $C$ to $Y$ with $f(q) = p$, $\xi \in f_* T_{C,q}$ and $f^* \eta
> 0$ at $q$.
\end{quote}
Therefore, we can also exclude the curves contained in a fixed proper subvariety of
$Y$. So we may assume that $f(C)\not\subset \Delta\cup W$, where $W\subsetneq Y$ is the
subvariety such that
\begin{equation}\label{E813}
L \cdot \Gamma = 0 \text{ for a curve } \Gamma\Leftrightarrow \Gamma\subset W
\end{equation}
Such $W$ exists because $L$ is big and NEF (see \remref{REM902}). Let
\begin{equation}\label{E282}
\hat \CO_{C,q} \isom \BC[[t]]
\end{equation}
be the formal local ring of $C$ at $q$ and $\mu$ be its valuation, i.e., $\mu(t^n) =
n$. Let
\begin{equation}\label{E283}
\mu(f^* s_\Delta) = \lambda
\end{equation}
where $\Delta = \{s_\Delta = 0\}$. Then $\mu(f^* \delta_j) \ge 2\lambda$.
And since $\{\sigma_{0j}\}$ and hence $\{r_j\}$ are BPF at $p$, we have
\begin{equation}\label{E285}
f^*\left(\frac{s_1 r_j
  \ol{\delta}_j + \ol{s}_1 \ol{r}_j \delta_j + |\delta_j|^2}{(|s_0|^2 + |s_1|^2)\sum_j
  |r_j|^2}\right) = O(t^{2\lambda} + \ol{t}^{2\lambda} + |t|^{4\lambda})
\end{equation}
Therefore, we obtain
\begin{equation}\label{E286}
\left.\frac{\sqrt{-1}}{2\pi} f^* \partial \ol\partial \log \left(1 + \sum_j \frac{s_1 r_j
  \ol{\delta}_j + \ol{s}_1 \ol{r}_j \delta_j + |\delta_j|^2}{(|s_0|^2 + |s_1|^2)\sum_j
  |r_j|^2}\right)
\right|_{t=0} = 0
\end{equation}
by the Taylor expansion of the LHS. Consequently,
\begin{equation}\label{E287}
\begin{split}
f^*\eta \Big|_q &= \left.\frac{\sqrt{-1}}{2\pi} f^* \partial \ol\partial \log \left(\sum_j
|r_j|^2\right)\right|_q\\
&= \left.\frac{\sqrt{-1}}{2\pi} \partial \ol\partial \log \left(\sum_j
|f^* \sigma_{0j} |^2\right)\right|_q\\
\end{split}
\end{equation}
Since $H^0(m (1+\varepsilon) L + m \pi_B^* M - 2\Delta)\subset \Sigma_0$ and
$f(C)\not\subset \Delta\cup W$, the linear system $f^* \Sigma_0$ is big on $C$. Therefore,
$f^* \eta > 0$ at $q$ and $\eta > 0$ at $p$.
\end{proof}


\begin{thebibliography}{99}
\bibitem[M]{M} M. McQuillan, Old and new techniques in function field arithmetic,
preprint.

\bibitem[V1]{V1} P. Vojta, Diophantine inequalities and Arakelov theory, In S. Lang,
{\it Introduction to Arakelov Theory\/}, (1988), Springer-Verlag, 155-178.

\bibitem[V2]{V2} P. Vojta, On algebraic points on curves, {\it Compositio Mathematica\/}
{\bf 78} (1991), 29-36.

\bibitem[Y]{Y} K. Yamanoi, The second main theorem for small functions and related
  problems, {\it Acta. Math.\/} {\bf 192} (2004), 225-294.
\end{thebibliography}
\end{document}


\vfuzz0.5pc
\hfuzz0.5pc 

\newcommand{\claimref}[1]{Claim \ref{#1}}
\newcommand{\thmref}[1]{Theorem \ref{#1}}
\newcommand{\propref}[1]{Proposition \ref{#1}}
\newcommand{\lemref}[1]{Lemma \ref{#1}}
\newcommand{\coref}[1]{Corollary \ref{#1}}
\newcommand{\remref}[1]{Remark \ref{#1}}
\newcommand{\conjref}[1]{Conjecture \ref{#1}}
\newcommand{\questionref}[1]{Question \ref{#1}}
\newcommand{\defnref}[1]{Definition \ref{#1}}
\newcommand{\secref}[1]{Sec. \ref{#1}}
\newcommand{\ssecref}[1]{\ref{#1}}
\newcommand{\sssecref}[1]{\ref{#1}}

\newcommand{\red}{{\mathrm{red}}}
\newcommand{\tors}{{\mathrm{tors}}}
\newcommand{\eq}{\Leftrightarrow}

\newcommand{\mapright}[1]{\smash{\mathop{\longrightarrow}\limits^{#1}}}
\newcommand{\mapleft}[1]{\smash{\mathop{\longleftarrow}\limits^{#1}}}
\newcommand{\mapdown}[1]{\Big\downarrow\rlap{$\vcenter{\hbox{$\scriptstyle#1$}}$}}
\newcommand{\smapdown}[1]{\downarrow\rlap{$\vcenter{\hbox{$\scriptstyle#1$}}$}}

\newcommand{\A}{{\mathbb A}}
\newcommand{\I}{{\mathcal I}}
\newcommand{\J}{{\mathcal J}}
\newcommand{\CO}{{\mathcal O}}
\newcommand{\CC}{{\mathcal C}}
\newcommand{\C} {{\mathbb C}}
\newcommand{\BC}{{\mathbb C}}
\newcommand{\BQ}{{\mathbb Q}}
\newcommand{\m}{{\mathcal M}}
\newcommand{\h}{{\mathcal H}}
\newcommand{\ZZ}{{\mathcal Z}}
\newcommand{\Z} {{\mathbb Z}}
\newcommand{\BZ}{{\mathbb Z}}
\newcommand{\W}{{\mathcal W}}
\newcommand{\Y}{{\mathcal Y}}
\newcommand{\T}{{\mathcal T}}
\newcommand{\BP}{{\mathbb P}}
\newcommand{\CP}{{\mathcal P}}
\newcommand{\G}{{\mathbb G}}
\newcommand{\BR}{{\mathbb R}}
\newcommand{\D}{{\mathcal D}}
\newcommand {\LL} {{\mathcal L}}
\newcommand{\f}{{\mathcal F}}
\newcommand{\E}{{\mathcal E}}
\newcommand{\BN}{{\mathbb N}}
\newcommand{\N}{{\mathcal N}}
\newcommand{\K}{{\mathcal K}}
\newcommand{\R} {{\mathbb R}}
\newcommand {\PP} {{\mathbb P}}
\newcommand{\BF}{{\mathbb F}}
\newcommand{\closure}[1]{\overline{#1}}
\newcommand{\EQ}{\Leftrightarrow}
\newcommand{\imply}{\Rightarrow}
\newcommand{\isom}{\cong}
\newcommand{\embed}{\hookrightarrow}
\newcommand{\tensor}{\mathop{\otimes}}
\newcommand{\wt}[1]{{\widetilde{#1}}}
\newcommand{\ol}{\overline}
\newcommand{\ul}{\underline}
\newcommand{\QQ}{{\mathcal Q}}

\newcommand{\bs}{{\backslash}}
\newcommand{\CS}{{\mathcal S}}
\newcommand{\Q} {{\mathbb Q}}
\newcommand {\CF} {{\mathcal F}}
\newcommand {\CG} {{\mathcal G}}

\title{Erratum to ``On Vojta's $1+\varepsilon$ Conjecture''}

\author{Xi Chen}

\address{632 Central Academic Building\\
University of Alberta\\
Edmonton, Alberta T6G 2G1, CANADA}

\email{xichen@math.ualberta.ca}

\date{\today}

\thanks{Research partially supported by a grant from the Natural Sciences
and Engineering Research Council of Canada.}

\keywords{Vojta's conjecture, Arithmetic geometry}

\subjclass{Primary 14J28; Secondary 14E05}

\begin{abstract}
We correct two mistakes in the paper ``On Vojta's $1+\varepsilon$ Conjecture''.
\end{abstract}

\maketitle

\section{Introduction}\label{SEC000INTRO}

The purpose of this note is to correct two mistakes in \cite{C}, kindly pointed out to me by Paul Vojta.

The main issue lies in the proof of Proposition 3.5 \cite[p. 99]{C}. It made the elementary mistake that $\CO_Y/I_\Delta^2$
is a coherent sheaf over $\Delta$. Indeed, there is no canonical way to make $A/I^2$ into a module over
$A/I$ for an ideal $I$ of a ring $A$.

Another issue lies at the very end of the paper \cite[p. 103]{C} where we need the positivity of
the $(1,1)$ form
\begin{equation}\label{E000743}
\frac{\sqrt{-1}}{2\pi} \partial \overline{\partial} \log \left(\sum_j \log |f^* \sigma_{0j}|^2\right)
\end{equation}
at a point $q$ for all immersions $f:C \to Y$ with $p = f(q)\in \Delta\backslash (S_0 \cup Y_Q)$. 
The fact that the linear series $f^*\Sigma_0$ is big on $C$ is not sufficient for \eqref{E000743} to be positive,
as the original argument went. In order for it to be positive,
the map $C\to \PP^N$ given by the linear series $f^*\Sigma_0$ has to be immersive at $q$,
which is not guaranteed by the bigness of $f^*\Sigma_0$.

These two issues are actually related. Proposition 3.5 says that $\Sigma_0$ is base points free at every point
$p\in \Delta\backslash (S_0\cup Y_Q)$; the positivity of \eqref{E000743} further requires it to
generate 1-jets at $p$. In the end, we can fix these problems provided that we
can give a correct proof of the following generalization of Prop. 3.5:

\begin{prop}\label{PROP000800}
For each non-negative integer $\mu$, 
there exist positive integers $m_0$ and $m_1$, depending only on
$\varepsilon$ and $\mu$, such that 
the linear system $\Sigma_i$ generates $\mu$-jets at every point $p\in \Delta\backslash (S_i\cup Y_Q)$, i.e., the map
\begin{equation}\label{E000750}
\begin{split}
&\quad
\Sigma_i = H^0(\CO_Y(m (1+\varepsilon) L + m \pi_B^* M)\tensor I_{F_i})\\
&\xrightarrow{}
H^0(\CO_Y(m (1+\varepsilon) L + m \pi_B^* M)\tensor I_{F_i}\tensor \CO_Y/I_p^{\mu+1})
\end{split}
\end{equation}
is surjective for all $m\in \BZ^+$, $m_0 | m$ and $\deg M \ge m_1$, where $I_p$ is the ideal sheaf of $p$ in $Y$.
\end{prop}

\subsection*{Acknowledgements}

I am very grateful to Prof. Paul Vojta, who not only spotted the errors in the original paper, as mentioned
above, but also helped me iron out the details of the fix presented here.

\section{Proof of \propref{PROP000800}}

\subsection{Notations}

For readers' convenience, we first recall all the necessary notations and results in \cite{C}.

Here $B$ is a smooth projective curve, $X = \PP^1\times B$ and $D$ is a union of $\alpha + 2 \ge 3$
sections of $X/B$ satisfying
\begin{itemize}
\item $D$ has at worst ordinary nodes as singularities, denoted by $D_{\sing}$;
\item $D$ has no more than one node on each fiber $X_b$ of $X/B$.
\end{itemize}

Let $Y = \PP T_X(-\log D)$, $L$ be the corresponding tautological line bundle, $\Delta\subset Y$ be the closure of the
image $\PP(T_{X/B}(-D)) \dashrightarrow Y$ and
\begin{equation}\label{E000749}
G = (1+\varepsilon) L + \pi_B^* M - \pi_X^* ( \omega_{X/B} + D)
\end{equation}
where $\varepsilon\in \BQ^+$, $M$ is a sufficiently ample divisor on $B$ and $\pi, \pi_X$ and $\pi_B$ are
the projections $\pi: X\to B$, $\pi_X: Y\to X$ and $\pi_B = \pi\circ \pi_X: Y\to B$, respectively.
Let $Q = \pi(D_{\sing})$ and $Y_Q = \pi_B^{-1}(Q)$.

Let $m$ be a positive integer such that $m\varepsilon \in \BZ$ and $S_i\subset Y$ be a union of $m\alpha$ general fibers
of $Y/\PP^1$ for $i=0,1$ satisfying $S_i\cap \pi_X^{-1}(D_{\sing}) = \emptyset$ and $S_0\cap S_1 = \emptyset$.

Finally, $F_i\subset Y$ is the subscheme $S_i\cap 2\Delta$ and $\Sigma_i$ is the linear system of the global sections of
$m (1+\varepsilon) L + m \pi_B^* M$ vanishing on $F_i$, i.e., defined as in \eqref{E000750}.

Recall that $\pi_X: \Delta\to X$ is the blowup of $X$ along $D_{\sing}$ \cite[p. 95]{C}
with $E = \sum_{q\in Q} E_q$ the exceptional divisor, where $E_q\subset \Delta_q = \Delta\cap Y_q$.
The way we construct $(X, D)$ guarantees that we blow up exactly one point on
the fiber $X_q$ for each $q\in Q$. 
So the fiber $\Delta_q$ is the union of two smooth rational curves meeting
transversely at one point; one is $E_q$ and we call the other component $E_q'$, i.e.,
$\Delta_q = E_q\cup E_q'$; both $E_q$ and $E_q'$ are $(-1)$-curves on $\Delta$. Note that
$E_q\cap S_i = \emptyset$ for all $q\in Q$.

Here we use a slightly different convention from \cite{C}, where we indexed $E_q$ by $q\in D_{\sing}$. 

Numerically, we have \cite[p. 95-96]{C}
\begin{equation}\label{E000748}
L\bigg|_\Delta = \pi_X^*(\omega_{X/B} + D) - E,\ \Delta = L - \pi_B^* \omega_B
\end{equation}
and
\begin{equation}\label{E000700}
G\bigg|_\Delta = \varepsilon \pi_X^*(\omega_{X/B} + D) + \pi_B^* M - (1+\varepsilon) E.
\end{equation}
It follows that
\begin{equation}\label{E000701}
G\cdot E_q = 1+\varepsilon \text{ and }
G\cdot E_q' = \varepsilon \alpha - (1+\varepsilon).
\end{equation}
Note that $G\cdot E_q' < 0$ for $\varepsilon << 1$ and hence $E_q'$ is where
$G$ fails to be nef when restricted to $\Delta$. Indeed, we have
\begin{equation}\label{E000702}
\bigcap_{m>0} \Bs(H^0(\Delta, mG)) = \bigcup_{q\in Q} E_q' = E'
\end{equation}
for $\varepsilon << 1$ and $M$ sufficiently ample.

\subsection{Proof of Proposition 3.5}

Before we proceed onto \propref{PROP000800},
let us first analyze the flawed proof of Proposition 3.5 in \cite{C} to show where the problem is.

For convenience, we make
\begin{equation}\label{E000780}
m G = m (1+\varepsilon) L + m \pi_B^* M - S_i
\end{equation}
by modifying the divisor $M\subset B$.

The proof of Proposition 3.5 starts with the exact sequence \cite[p. 100]{C}
\begin{equation}\label{E000703}
\begin{split}
0 &\xrightarrow{} \underbrace{\CO_Y(m (1+\varepsilon) L + m \pi_B^* M - 2\Delta)}_{\CO_Y(mG + S_i - 2\Delta)} \\
& \xrightarrow{}
\underbrace{\CO_Y(m (1+\varepsilon) L + m \pi_B^* M)}_{\CO_Y(mG + S_i)} \tensor I_{F_i}\\
& \xrightarrow{}
\underbrace{\CO_Y(m (1+\varepsilon) L + m \pi_B^* M - S_i)}_{\CO_Y(mG)} \tensor \CO_Y/I_\Delta^2 
\xrightarrow{} 0
\end{split}
\end{equation}
followed by the surjection \cite[(3.20), p. 100]{C}
\begin{equation}\label{E000704}
\begin{split}
\Sigma_i & \twoheadrightarrow H^0(\CO_Y(m (1+\varepsilon) L + m \pi_B^* M - S_i)
\tensor \CO_Y/I_\Delta^2)\\
&\quad \isom H^0(\CO_Y(mG) \tensor \CO_Y/I_\Delta^2)
\end{split}
\end{equation}
since
\begin{equation}\label{E000708}
H^1(Y, m (1+\varepsilon) L + m \pi_B^* M - 2\Delta) = 0
\end{equation}
by Lemma 3.2 \cite[p. 96]{C}. It should be mentioned here that Lemma 3.2 only holds for $m\in \BZ^+$ and $m\ge k\ge 0$.

So far everything is fine except that $\CO_\Delta(mG)$ in the proof is here replaced by $\CO_Y(mG)$
since $\CO_Y/I_\Delta^2$ is not a coherent sheaf over $\Delta$, as pointed out at the very beginning.

However, the surjectivity of the top row of the diagram \cite[p.100]{C}
\begin{equation}\label{E000705}
\xymatrix{
H^0(\CO_Y(mG) \tensor \CO_Y/I_\Delta^2)\ar[d]^\varphi \ar@{>>}[r] & H^0(Y_b, \CO_Y(mG)
\tensor \CO_Y/I_\Delta^2) \ar[d]^{\varphi_b}\\
H^0(\CO_\Delta(mG)) \ar@{>>}[r] & H^0(\Delta_b, \CO_\Delta(mG))}
\end{equation}
for $b\in B\backslash Q$ cannot be easily established. The problem is, again, that $\CO_Y/I_\Delta^2$ cannot be
regarded as a coherent sheaf over $\Delta$.

Let $\f$ be a coherent sheaf on $\Delta$. Then
we always have a surjection
\begin{equation}\label{E000760}
H^0(\Delta, \CO(mG) \tensor \f) \twoheadrightarrow H^0(\Delta_b, \CO(mG)\tensor \f)
\end{equation}
for $m$ sufficiently large and $b\not\in Q$ by the following argument: by \eqref{E000702}, we can
write $G = G_1 + G_2$ on $\Delta$ with $G_1$ ample and $G_2$ effective and supported on $E'$; then
the surjectivity of \eqref{E000760} follows from the diagram
\begin{equation}\label{E000732}
\xymatrix{H^0(\Delta, \CO(mG_1)\tensor \f) \ar[r]^\subset\ar@{>>}[d] & H^0(\Delta, \CO(mG)\tensor \f)\ar@{>>}[d]\\
H^0(\Delta_b, \CO(mG_1) \tensor \f) \ar@{=}[r] & H^0(\Delta_b, \CO(mG) \tensor \f)
}
\end{equation}
since $G_2\cap \Delta_b = \emptyset$.

\subsection{Generation of $\mu$-jets}

We have to fix the proof of Proposition 3.5 and generalize it to $\mu$-jets.

Let $\eta: \widehat{Y}\to Y$ be the blowup of $Y$ at $p\in \Delta\backslash (S_i\cup Y_Q)$ with exceptional divisor
$J\isom \PP^2$. Then $\Sigma_i$ generates $l$-jets at $p$ if and only if the map
\begin{equation}\label{E000747}
\begin{split}
\widehat{\Sigma}_i 
&= H^0(\CO_{\widehat{Y}}(\eta^*(m (1+\varepsilon) L + m \pi_B^* M) - \mu J)\tensor I_{\widehat{F}_i})
\\
&= H^0(\CO_{\widehat{Y}} (m\widehat{G} + \widehat{S}_i - \mu J)
\tensor I_{\widehat{F}_i})\\
& \xrightarrow{} H^0(J, \CO(-\mu J)) = H^0(J, \CO(\mu))
\end{split}
\end{equation}
is surjective for all $0\le \mu\le l$,
where $\widehat{\Delta}\subset \widehat{Y}$ is the proper transform of $\Delta$ under $\eta$, 
$\widehat{G} = \eta^* G$, $\widehat{S}_i
= \eta^* S_i$ and $\widehat{F}_i = \widehat{S}_i \cap 2 \widehat{\Delta}$.

From a similar exact sequence 
\begin{equation}\label{E000761}
\begin{split}
0 &\xrightarrow{} \CO_{\widehat{Y}}( m\widehat{G} + \widehat{S}_i - 2\widehat{\Delta} - \mu J) \\
& \xrightarrow{}
\CO_{\widehat{Y}} (m\widehat{G} + \widehat{S}_i - \mu J)
\tensor I_{\widehat{F}_i}\\
& \xrightarrow{}
\CO_{\widehat{Y}}(m\widehat{G} - \mu J) 
\tensor \CO_{\widehat{Y}}/I_{\widehat{\Delta}}^2 
\xrightarrow{} 0
\end{split}
\end{equation}
to \eqref{E000703}, we obtain the diagram
\begin{equation}\label{E000762}
\xymatrix{
H^0(\CO_{\widehat{Y}}( m\widehat{G} + \widehat{S}_i - 2\widehat{\Delta} - \mu J)
\ar[r]\ar[d] & H^0(J, \CO(\mu - 2))\ar[d]\\
H^0(\CO_{\widehat{Y}} (m\widehat{G} + \widehat{S}_i - \mu J)\tensor I_{\widehat{F}_i})
\ar[r]\ar[d] & H^0(J, \CO(\mu)) \ar[d]\\
H^0(\CO_{\widehat{Y}}(m\widehat{G} - \mu J) \tensor \CO_{\widehat{Y}}/I_{\widehat{\Delta}}^2)
\ar[r] & H^0(J, \CO(\mu)\tensor \CO_{\widehat{Y}}/I_{\widehat{\Delta}}^2)
}
\end{equation}
by restricting it to $J$ and taking the long exact sequences.

\begin{lem}\label{LEM000702}
There exists $\lambda\in \BR^+$, independent of $p$, such that
\begin{equation}\label{E000746}
H^1(\widehat{Y}, \eta^*(m L + \pi_B^* M) - a J) = 0
\end{equation}
for all $m\in \BZ^+$, $\deg M \ge \lambda m$, $a\in \BZ$ and $a\le m$.
\end{lem}

\begin{proof}
From the exact sequence
\begin{equation}\label{E000745}
\begin{split}
& \quad H^1(\widehat{Y}, \eta^*(m L + \pi_B^* M) - (a+1) J)\\
 & \xrightarrow{} H^1(\widehat{Y}, \eta^*(m L + \pi_B^* M) - a J)
\xrightarrow{}
H^1(J, \CO(a)) = 0
\end{split}
\end{equation}
we see that it suffices to prove \eqref{E000746} for $a = m$.

Following the proof of Lemma 3.2 \cite[p. 96]{C}, we first show that
\begin{equation}\label{E000744}
H^1(\widehat{Y}_b, m\eta^*L - m J) = 0
\end{equation}
for all $b\in B$. When $p\not\in Y_b$, this follows from the argument for Lemma 3.2.

If $p\in Y_b$, $\widehat{Y}_b$ is the union $T\cup J$, where $T$ is the proper transform of $Y_b$ under $\eta$. We have the
exact sequence
\begin{equation}\label{E000752}
0 \xrightarrow{} \CO_{T\cup J} \xrightarrow{} \CO_T \oplus \CO_{J} \xrightarrow{}
\CO_{T\cap J} \xrightarrow{} 0
\end{equation}
on $T\cup J$, which induces the long exact sequence
\begin{equation}\label{E000764}
\begin{split}
&\quad H^0(\CO_T(m\eta^* L - mJ)) \oplus H^0(\CO_J(-mJ)) \xrightarrow{}
H^0(\CO_{T\cap J}(-mJ))\\
&\xrightarrow{}
H^1(\CO_{T\cup J}(m\eta^* L - mJ))
\\
&\xrightarrow{}
H^1(\CO_T(m\eta^* L - mJ)) \oplus H^1(\CO_J(-mJ)).
\end{split}
\end{equation}
It is easy to check that $H^1(\CO_T(m\eta^* L - mJ)) = H^1(\CO_J(-mJ)) = 0$ and
$H^0(\CO_J(-mJ)) \twoheadrightarrow H^0(\CO_{T\cap J}(-mJ))$ is a surjection. And \eqref{E000744} follows.

Therefore, $R^1 (\pi_B\circ \eta)_* \CO_{\widehat{Y}} (\eta^*(m L + \pi_B^* M) - m J) = 0$ and hence
\begin{equation}\label{E000753}
\begin{split}
&\quad H^1(\widehat{Y}, \eta^*(m L + \pi_B^* M) - m J)\\
& = H^1(B, (\pi_B\circ \eta)_* \CO_{\widehat{Y}} (\eta^*(m L + \pi_B^* M) - m J))
\\
&= H^1(B, (\pi_B\circ \eta)_* \CO_{\widehat{Y}} (m \eta^*L - m J) \tensor \CO_B(M)).
\end{split}
\end{equation}

Obviously, there exists $\lambda \in \BR^+$ such that
\begin{equation}\label{E000765}
H^1(B, \Sym^m (\pi_B)_* \CO_{Y} (L) \tensor \CO_B(M)) = 0
\end{equation}
for all $m\in \BZ^+$ and $\deg M \ge (\lambda - 1) m$. Note that such $\lambda$ does not depend on $p$.

We observe that if $\CF\to \CG$ is a map of coherent sheaves on $B$ which is surjective at a general
point of $B$ and $H^1(B, \CF) = 0$, then $H^1(B, \CG) = 0$.
Applying this to the map
\begin{equation}\label{E000766}
\begin{split}
(\pi_B\circ \eta)_* \CO_{\widehat{Y}} (\eta^*L) & \xrightarrow{} (\pi_B\circ \eta)_*
\CO_{\widehat{Y}} (\eta^*L + T)\\
&\quad = (\pi_B\circ \eta)_* (\eta^*L -  J) \tensor \CO_B(b),
\end{split}
\end{equation}
we conclude that
\begin{equation}\label{E000754}
H^1(B, \Sym^m (\pi_B\circ \eta)_* \CO_{\widehat{Y}} (\eta^*L -  J) \tensor \CO_B(M)) = 0
\end{equation}
for all $m\in \BZ^+$ and $\deg M \ge \lambda m$.

On the other hand, the natural map
\begin{equation}\label{E000755}
\Sym^m (\pi_B\circ \eta)_* \CO_{\widehat{Y}} (\eta^*L -  J)
\xrightarrow{}
(\pi_B\circ \eta)_* \CO_{\widehat{Y}} (m\eta^*L -  mJ)
\end{equation}
is surjective at a general point of $B$. It then follows from \eqref{E000754} that
\begin{equation}\label{E000756}
H^1(B, (\pi_B\circ \eta)_* \CO_{\widehat{Y}} (m \eta^*L - m J) \tensor \CO_B(M)) = 0.
\end{equation}
And we obtain \eqref{E000746} by \eqref{E000753}.
\end{proof}

By the above lemma, 
\begin{equation}\label{E000709}
\begin{split}
&\quad
H^1(\CO_{\widehat{Y}}( m\widehat{G} + \widehat{S}_i - 2\widehat{\Delta} - \mu J)
\\
&= H^1(\CO_{\widehat{Y}}( \eta^*(m (1+\varepsilon) L + m \pi_B^* M) - 2\widehat{\Delta} - \mu J)
\\
&= H^1(\CO_{\widehat{Y}}( m\widehat{G} + \widehat{S}_i - 2\widehat{\Delta} - (\mu+1) J)
\\
&= H^1(\CO_{\widehat{Y}}( \eta^*(m (1+\varepsilon) L + m \pi_B^* M) - 2\widehat{\Delta} - (\mu+1) J)
=0
\end{split}
\end{equation}
for $m(1+\varepsilon) > \mu$ and $\deg M >> 1$.
Therefore, the first column of \eqref{E000762} is both left and right exact and
its first row is surjective. Thus, its middle row, i.e., \eqref{E000747} is surjective if its third row is. That is,
it suffices to prove that the map
\begin{equation}\label{E000763}
H^0(\CO_{\widehat{Y}}(m\widehat{G} - \mu J) \tensor \CO_{\widehat{Y}}/I_{\widehat{\Delta}}^2)
\xrightarrow{} H^0(J, \CO(\mu)\tensor \CO_{\widehat{Y}}/I_{\widehat{\Delta}}^2)
\end{equation}
is surjective. One might try to argue as follows but there are some obstacles to overcome.

From the exact sequence \cite[p. 100]{C}
\begin{equation}\label{E000706}
\xymatrix{0 \ar[r] & I_{\widehat{\Delta}} /I_{\widehat{\Delta}}^2 \ar[r]\ar@{=}[d] & \CO_{\widehat{Y}}/
I_{\widehat{\Delta}}^2 \ar[r]
& \CO_{\widehat{Y}}/I_{\widehat{\Delta}} \ar[r] \ar@{=}[d] & 0\\
& \CO_{\widehat{\Delta}}(-\widehat{\Delta}) & & \CO_{\widehat{\Delta}}}
\end{equation}
we have the diagram
\begin{equation}\label{E000707}
\xymatrix{
H^0(\widehat{\Delta}, m \widehat{G} - \mu J - \widehat{\Delta}) \ar[d]\ar[r] & H^0(\widehat{\Delta}\cap J, \CO(\mu - 1))
\ar[d]\\
H^0(\CO_{\widehat{Y}}(m\widehat{G} - \mu J) \tensor \CO_{\widehat{Y}}/I_{\widehat{\Delta}}^2) \ar[d] \ar[r] & 
H^0(J, \CO(\mu)\tensor \CO_{\widehat{Y}}/I_{\widehat{\Delta}}^2) 
\ar[d]\\
H^0(\widehat{\Delta}, m \widehat{G} - \mu J) \ar[r] & H^0(\widehat{\Delta}\cap J, \CO(\mu))
}
\end{equation}
where the right column is obtained by restricting \eqref{E000706} to $J$. Note that
$\widehat{\Delta}$ and $J$ meet properly in $\widehat{Y}$ and hence
$\operatorname{\mathscr T\kern-1.5pt{\text{or}}}_{\scriptscriptstyle 1}(\CO_{\widehat{Y}}/ I_{\widehat{\Delta}}, \CO_{\widehat{Y}}/ I_{J}) = 0$. Therefore,
\eqref{E000706} remains exact when restricted to $J$. That is, we have the exact sequence
\begin{equation}\label{E000784}
\xymatrix{0 \ar[r] & \CO_{\widehat{\Delta}\cap J}(-\widehat{\Delta}) \ar[r]\ar@{=}[d]
& \CO_{\widehat{Y}}/
I_{\widehat{\Delta}}^2 \bigg|_J \ar[r]
& \CO_{\widehat{\Delta}\cap J} \ar[r]  & 0\\
& \CO_{\PP^1}(-1) }
\end{equation}
where $\widehat{\Delta}\cap J\isom \PP^1$ is a line in $J\isom\PP^2$. Then
the right column of \eqref{E000707} is both exact on the left and in the middle. Indeed, although we do not need it, it is also exact on the right since
$H^1(\widehat{\Delta}\cap J, \CO(\mu - 1)) = H^1(\PP^1, \CO(\mu-1)) = 0$ for
$\mu \ge 0$.

Again, the middle arrow of \eqref{E000707} is a surjection if its first column is both left and right exact and its first and
third rows are surjective; namely, we need 
\begin{equation}\label{E000723}
\begin{split}
H^1(\widehat{\Delta}, m\widehat{G} - \mu J - \widehat{\Delta}) &=
H^1(\widehat{\Delta}, m \widehat{G} - (\mu+1) J)\\
&=
H^1(\widehat{\Delta}, m\widehat{G} - (\mu+1) J - \widehat{\Delta}) = 0.
\end{split}
\end{equation}
Unfortunately, \eqref{E000723} fails to hold
due to the fact that $G$ is not ample on $\Delta$, i.e., \eqref{E000701} and \eqref{E000702}. 
So this argument breaks down.
However, we can make it work by birationally ``modifying'' $Y$ to ``render'' $G$ ample on $\Delta$. 

\subsection{Birational modification}

Let $\nu: \wt{Y}\to \widehat{Y}$ be the blowup of $\widehat{Y}$ along $\widehat{E}' = \eta^{-1}(E')$. 
That is, $\xi = \eta\circ \nu: \wt{Y}\to Y$ is the blowup of $Y$ along the curve $E'$ and the point $p$. 

Let $\wt{\Delta}$ be the proper transform of $\Delta$
under $\xi$ and let $\E_q'\subset \wt{Y}$ be the exceptional divisor over $E_q'$.
We can determine the type of rational ruled surfaces we have for $\E_q' = \PP N_{E_q'/Y}$ as follows.
The inclusions $E_q' \subset \Delta \subset Y$ give us the exact sequence of normal bundles
\begin{equation}\label{E000710}
\xymatrix{
0\ar[r] & N_{E_q'/\Delta} \ar[r] \ar@{=}[d] & N_{E_q'/Y} \ar[r] & N_{\Delta/Y}\bigg|_{E_q'} \ar[r] \ar@{=}[d] & 0\\
        & \CO(-1)                           &                   &         \CO(\alpha - 1)                     &
}
\end{equation}
where $N_{\Delta/Y}\big|_{E_q'} = \CO(\alpha - 1)$ since
$\Delta \cdot E_q' = \alpha -1$ by \eqref{E000748}.
To determine how \eqref{E000710} splits, we also notice that $E_q'$ lies on the fiber $Y_q$ of $Y/B$ over $q\in B$. So the exact sequence
of normal bundles derived from the inclusions $E_q' \subset Y_q\subset Y$ gives us a surjection $N_{E_q'/Y} \to \CO_{E_q'}$. Therefore,
$N_{E_q'/Y}$ must be
\begin{equation}\label{E000711}
N_{E_q'/Y} = \CO \oplus \CO(\alpha - 2)
\end{equation}
and hence $\E_q' \isom \BF_{|\alpha-2|}$.

Let
\begin{equation}\label{E000712}
\wt{G} = \xi^* G - \delta \E' = 
\xi^* G - \delta \sum_{q\in Q} \E_q'
\end{equation}
where we choose $\delta > 0$ such that 
\begin{equation}\label{E000717}
(1+\varepsilon) - \varepsilon \alpha < \delta < 1+ \varepsilon.
\end{equation}
Let $\wt{E}_q$ and $\wt{E}_q'\subset \wt{\Delta}$ be the proper
transforms of $E_q$ and $E_q'$ under $\wt{\Delta}\to \Delta$, respectively.
Then
\begin{equation}\label{E000713}
\wt{G}\cdot \wt{E}_q = 1+ \varepsilon - \delta > 0 \text{ and }
\wt{G} \cdot \wt{E}_q' = \delta - (1+\varepsilon) + \varepsilon \alpha > 0
\end{equation}
for all $q\in Q$. Consequently, $m \wt{G} - a\wt{J}$ is ample on $\wt{\Delta}$ for $m\varepsilon\alpha > a > 0$,
where $\wt{J} = \nu^* J$.

Let $\wt{S}_i = \xi^* S_i$ and let $\wt{F}_i\subset \wt{Y}$ be the subscheme $\wt{S}_i\cap 2\wt{\Delta}_i$.
Note that $S_i$ is disjoint from $E$ as mentioned before.

Applying the same argument for $\widehat{Y}$ to $\wt{Y}$, we obtain the exact sequence
\begin{equation}\label{E000714}
\begin{split}
0 &\xrightarrow{} \CO_{\wt{Y}}(m \wt{G} + \wt{S}_i - 2\wt{\Delta} - \mu \wt{J})
\\
& \xrightarrow{}
\CO_{\wt{Y}}(m \wt{G} + \wt{S}_i - \mu \wt{J}) \tensor I_{\wt{F}_i}\\
& \xrightarrow{}
\CO_{\wt{Y}}(m\wt{G} - \mu \wt{J}) \tensor \CO_{\wt{Y}}/I_{\wt{\Delta}}^2 
\xrightarrow{} 0
\end{split}
\end{equation}
on $\wt{Y}$, similar to \eqref{E000761}, and the diagrams
\begin{equation}\label{E000770}
\xymatrix{
H^0(\CO_{\wt{Y}}( m\wt{G} + \wt{S}_i - 2\wt{\Delta} - \mu \wt{J})
\ar[r]\ar[d] & H^0(\wt{J}, \CO(\mu - 2))\ar[d]\\
H^0(\CO_{\wt{Y}} (m\wt{G} + \wt{S}_i - \mu \wt{J})\tensor I_{\wt{F}_i})
\ar[r]\ar[d] & H^0(\wt{J}, \CO(\mu)) \ar[d]\\
H^0(\CO_{\wt{Y}}(m\wt{G} - \mu \wt{J}) \tensor \CO_{\wt{Y}}/I_{\wt{\Delta}}^2)
\ar[r] & H^0(\wt{J}, \CO(\mu)\tensor \CO_{\wt{Y}}/I_{\wt{\Delta}}^2)
}
\end{equation}
and
\begin{equation}\label{E000771}
\xymatrix{
H^0(\wt{\Delta}, m \wt{G} - \mu \wt{J} - \wt{\Delta}) \ar[d]\ar[r] & H^0(\wt{\Delta}\cap \wt{J}, \CO(\mu - 1))
\ar[d]\\
H^0(\CO_{\wt{Y}}(m\wt{G} - \mu \wt{J}) \tensor \CO_{\wt{Y}}/I_{\wt{\Delta}}^2) \ar[d] \ar[r] & 
H^0(\wt{J}, \CO(\mu)\tensor \CO_{\wt{Y}}/I_{\wt{\Delta}}^2) 
\ar[d]\\
H^0(\wt{\Delta}, m \wt{G} - \mu \wt{J}) \ar[r] & H^0(\wt{\Delta}\cap \wt{J}, \CO(\mu))
}
\end{equation}
analogous to \eqref{E000762} and \eqref{E000707}.

Observe that
\begin{equation}\label{E000716}
\begin{split}
\wt{\Sigma}_i &= H^0(\wt{Y}, \CO_{\wt{Y}}(m \wt{G} + \wt{S}_i - \mu \wt{J}) \tensor I_{\wt{F}_i})
\\
&\hookrightarrow
H^0(\wt{Y}\backslash \E', \CO_{\wt{Y}}(m\wt{G} + \wt{S}_i - \mu \wt{J}) \tensor I_{\wt{F}_i})\\
&= H^0(\widehat{Y}\backslash \widehat{E}', \CO_{\widehat{Y}} (m \widehat{G} + \widehat{S}_i 
- \mu J)\tensor I_{\widehat{F}_i})\\
&= H^0(\widehat{Y}, \CO_{\widehat{Y}}(m \widehat{G} + \widehat{S}_i - \mu J)\tensor I_{\widehat{F}_i}) 
= \widehat{\Sigma}_i
\end{split}
\end{equation}
where
\begin{equation}\label{E000718}
\begin{split}
&\quad H^0(\widehat{Y}\backslash \widehat{E}', \CO_{\widehat{Y}}(m \widehat{G} + \widehat{S}_i
-\mu J)\tensor I_{\widehat{F}_i})\\
&= 
H^0(\widehat{Y}, \CO_{\widehat{Y}}(m \widehat{G} + \widehat{S}_i
-\mu J)\tensor I_{\widehat{F}_i})
\end{split}
\end{equation}
since $\widehat{E}'= \cup \widehat{E}_q'$ has codimension $2$ in $\widehat{Y}$, 
$\dim(\widehat{E}'\cap \widehat{F}_i) = 0$ and
$\widehat{F}_i$ is of pure dimension $1$.
In conclusion, $\wt{\Sigma}_i$ injects into
$\widehat{\Sigma}_i$. Combining this with the diagrams \eqref{E000762} and \eqref{E000770}, we obtain
\begin{equation}\label{E000719}
\xymatrix{
\wt{\Sigma}_i \ar[d]^{\subset} \ar[r] & H^0(\CO_{\wt{Y}}(m\wt{G} - \mu \wt{J}) \tensor \CO_{\wt{Y}}/I_{\wt{\Delta}}^2)
\ar[r] & H^0(\wt{J}, \CO(\mu) \tensor \CO_{\wt{Y}}/I_{\wt{\Delta}}^2) \ar@{=}[d]\\
\widehat{\Sigma}_i \ar@{>>}[r] & H^0(\CO_{\widehat{Y}}(m\widehat{G} - \mu J) \tensor 
\CO_{\widehat{Y}}/I_{\widehat{\Delta}}^2) \ar[r] &  H^0(J, \CO(\mu) \tensor \CO_{\widehat{Y}}/I_{\widehat{\Delta}}^2).
}
\end{equation}
To show the surjectivity of \eqref{E000763}, it suffices to show that
both maps of the top row of \eqref{E000719} are surjections. 
Then by \eqref{E000770} and \eqref{E000771}, it comes down to verifying
\begin{equation}\label{E000715}
\begin{split}
&\quad H^1(\wt{Y}, m \wt{G} + \wt{S}_i - 2\wt{\Delta} - \mu \wt{J})
\\
&= H^1(\wt{Y}, \xi^*(m (1+\varepsilon) L + m \pi_B^* M) - m\delta \E' - 2\wt{\Delta} - \mu \wt{J}) = 0,
\end{split}
\end{equation}
\begin{equation}\label{E000724}
H^1(\wt{\Delta}, m \wt{G} - \mu \wt{J} - \wt{\Delta}) = 0,
\end{equation}
\begin{equation}\label{E000725}
H^1(\wt{\Delta}, m \wt{G} - (\mu+1) \wt{J}) = 0, \ \text{and}
\end{equation}
\begin{equation}\label{E000726}
H^1(\wt{\Delta}, m \wt{G} - (\mu+1) \wt{J} - \wt{\Delta}) = 0.
\end{equation}

Using the same argument for \lemref{LEM000702}, we can prove

\begin{lem}\label{LEM000700}
There exists $\lambda\in \BR^+$, independent of $p$, such that
\begin{equation}\label{E000731}
H^1(\wt{Y}, \xi^*(m L + \pi_B^* M) - a_1 \E' - a_2 \wt{J}) = 0
\end{equation}
for all $m\in \BZ^+$, $\deg M \ge \lambda m$, $a_1, a_2\in \BZ$, $0\le a_1 \le m$ and $a_2\le m$.
\end{lem}

\begin{proof}
First of all, it is enough to prove \eqref{E000731} for $a_2 = m$ as in \eqref{E000746}.

Since
\begin{equation}\label{E000772}
L\cdot E_q' = \alpha - 1,
\end{equation}
we have
\begin{equation}\label{E000721}
\CO_{\E_q'}(m \xi^* L - a_1 \E') = \CO_{\E_q'}(a_1) \tensor \rho^* \CO(m(\alpha - 1))
\end{equation}
for all $q\in Q$, where $\CO_{\E_q'}(1)$ is the tautological bundle of $\E_q' = \PP N_{E_q'/Y}$ and $\rho$ is the projection
$\E_q' \to E_q'\isom \PP^1$. Therefore,
\begin{equation}\label{E000729}
H^1(\E_q', m \xi^* L - a_1 \E') = 0
\end{equation}
for $0\le a_1\le m$ by \eqref{E000711}. 
Then using a similar exact sequence to \eqref{E000745}, we see that \eqref{E000731} holds for $0\le a_1\le m$ if it holds
for $a_1 = m$.
In conclusion, it is enough to prove \eqref{E000731} for $a_1 = a_2 = m$.

Next, we prove
\begin{equation}\label{E000734}
H^1(\wt{Y}_b, m \xi^* L - m \E' - m \wt{J}) = 0
\end{equation}
for all $b\in B$. For $b\not\in Q$, this is equivalent to \eqref{E000744}.

For $b = q\in Q$, the fiber $\wt{Y}_q$ is the union 
$\wt{Y}_q = T \cup \E_q'$, 
where $T \isom Y_q$ is the proper transform of $Y_q$ under $\xi$. Using a similar exact sequence to  
\eqref{E000752}, we reduce \eqref{E000734} to the vanishing of
\begin{equation}\label{E000733}
H^1(\E_q', m \xi^* L - m \E_q') = H^1(T, m\xi^* L - m\E_q') = 0
\end{equation}
and the surjectivity of
\begin{equation}\label{E000735}
H^0(T, m\xi^* L - m\E_q') \twoheadrightarrow H^0(T\cap \E_q', m\xi^* L- m\E_q').
\end{equation}
Note that
\begin{equation}\label{E000728}
L\bigg|_{Y_q} = \Delta \bigg|_{Y_q} = \Delta_q = E_q + E_q'
\end{equation}
and hence
\begin{equation}\label{E000722}
(\xi^* L- \E_q')\bigg|_{T} = \wt{E}_q
\end{equation}
where $\wt{E}_q\subset T$ is the proper transform of $E_q$ under $\xi: T\to Y_q$. Clearly, $E_q$ and $\wt{E}_q$ are fibers
of $\pi_X: Y\to X$ and $\pi_X\circ \xi: \wt{Y}\to X$, respectively.
Then \eqref{E000733} follows from \eqref{E000729} and \eqref{E000722}; the surjectivity of \eqref{E000735} follows from
\eqref{E000722} and the fact that $T\cap \E_q'$ is a section of $T/X_q$.

This proves \eqref{E000734} for every point $b\in B$. The rest of the proof goes
similarly to that of \lemref{LEM000702}. As \eqref{E000754}, we need to show that
\begin{equation}\label{E000900}
H^1(B, \Sym^m (\pi_B\circ \xi)_* \CO_{\wt{Y}} (\xi^*L -  \E' - \wt{J})
\tensor \CO_B(M)) = 0
\end{equation}
if $\deg M \ge \lambda m$ for some constant $\lambda > 0$. From the generically
surjective map
\begin{equation}\label{E000901}
(\pi_B\circ \xi)_* \CO_{\wt{Y}} (\xi^*L - \E') 
\xrightarrow{}
(\pi_B\circ \xi)_* \CO_{\wt{Y}} (\xi^*L - \E' - \wt{J}) \tensor \CO_B(b),
\end{equation}
as \eqref{E000766}, we see that it suffices to find $\lambda$ such that
\begin{equation}\label{E000902}
H^1(B, \Sym^m (\pi_B\circ \xi)_* \CO_{\wt{Y}} (\xi^*L -  \E')
\tensor \CO_B(M)) = 0
\end{equation}
for all $\deg M \ge (\lambda - 1)m$. Such $\lambda$ obviously exists independently of
$\wt{J}$, i.e., independently of the point $p$ where we blow up $Y$.
\end{proof}

Obviously, \lemref{LEM000700} implies
\begin{equation}\label{E000815}
H^1(\wt{Y}, \xi^*(m (1+\varepsilon) L
+ m \pi_B^* M) - m\delta \E' - 2\wt{\Delta} - \mu \wt{J\
}) = 0,
\end{equation}
which gives \eqref{E000715}. For
\eqref{E000724}-\eqref{E000726}, it suffices to prove the following

\begin{lem}\label{LEM000701}
There exists $\lambda\in \BR^+$, independent of $p$, such that
\begin{equation}\label{E000736}
H^1(\wt{\Delta}, \xi^*( m L + \pi_B^* M - a_3 \pi_X^*(\omega_{X/B}+D)) - a_1 \E' - a_2 \wt{J}) = 0
\end{equation}
for all $m\in \BZ^+$, $\deg M \ge \lambda m$ and $a_1, a_2, a_3\in \BZ$
satisfying
\begin{equation}\label{E000737}
0\le m - a_1 \le (m - a_3) \alpha\text{ and }
-1 \le a_2 \le (m - a_3)\alpha.
\end{equation}
\end{lem}

\begin{proof}
As before, it is enough to prove \eqref{E000736} for $a_1 = m$ and $a_2 = (m - a_3)\alpha$.

Then it is easy to see that there exists $\lambda > 0$ such that the divisor
\begin{equation}\label{E000738}
\begin{split}
&\quad \xi^*( m L + \pi_B^* M - a_3 \pi_X^*(\omega_{X/B}+D)) - m \E' - (m-a_3)\alpha \wt{J} - \omega_{\wt{\Delta}}
\\
&= \xi^*((m+1) L + \pi_B^* M - (a_3+1)\pi_X^*(\omega_{X/B}+D) - \pi_X^* \omega_X)\\
&\quad\quad - m \E' - ((m-a_3)\alpha + 1)\wt{J}
\end{split}
\end{equation}
is ample on $\wt{\Delta}$ for all $\deg M \ge \lambda m$.
\end{proof}